\documentclass{article}

\usepackage{fancyhdr, fullpage}
\usepackage{graphicx, array, subcaption}
\usepackage[american]{babel}
\usepackage[breaklinks,bookmarks,colorlinks,citecolor=blue]{hyperref}       
\usepackage{url}            
\usepackage{booktabs}       
\usepackage{amsfonts}       
\usepackage{nicefrac}       
\usepackage{microtype}      
\usepackage{xcolor}         
\usepackage[square,numbers]{natbib}
\usepackage{amssymb, amsmath, amsthm}
\usepackage{algorithm}
\usepackage{algpseudocode}
\algtext*{EndIf}
\usepackage{mathtools}
\usepackage{enumitem}
\usepackage{colortbl}
\usepackage{ragged2e}
\usepackage{multirow}
\usepackage{array, hhline}
\usepackage{subcaption, graphicx}

\newcommand{\R}{\mathbb{R}}
\newcommand{\cond}{\,\middle\vert\,}

\newcommand{\pf}[1]{\nabla f (#1)}

\newcommand{\pfik}[1]{\nabla f_{i_k}(#1)}

\newcommand{\pfd}[1]{\nabla f^{\delta_t} (#1)}
\newcommand{\pfdik}[1]{\nabla f^{\delta_t}_{i_k}(#1)}

\newcommand{\pfdii}[1]{\nabla f^{\delta_t}_{i}(#1)}

\newcommand{\ph}[1]{\nabla h (#1)}

\newcommand{\phd}[1]{\nabla h^{\delta_t} (#1)}
\newcommand{\phdik}[1]{\nabla h^{\delta_t}_{i_k}(#1)}

\newcommand{\phdii}[1]{\nabla h^{\delta_t}_{i}(#1)}

\newcommand{\norm}[1]{\left\lVert#1\right\rVert}

\newcommand{\innr}[1]{\left\langle#1\right\rangle}

\newcommand{\Eik}[1]{\mathbb{E}_{i_k} \left[#1\right]}

\newcommand{\E}[1]{\mathbb{E} \left[#1\right]}
\newcommand{\mleq}[1]{\overset{\mathclap{(#1)}}{\leq}}
\newcommand{\meq}[1]{\overset{\mathclap{(#1)}}{=}}

\newcommand{\mra}[1]{\overset{\mathclap{(#1)}}{\Rightarrow}}
\newcommand{\mar}[1]{\left(#1\right)}
\newcommand{\xs}{x^{\star}}
\newcommand{\xsd}{x^{\star}_{\delta_t}}
\newcommand{\Xs}{\mathcal{X}^\star}

\newcommand{\prob}[1]{\textup{Prob}\left\{#1\right\}}

\DeclareMathOperator*{\argmin}{arg\,min}

\definecolor{Gray}{gray}{0.9}

\newcolumntype{C}{>{\centering\let\newline\\\arraybackslash\hspace{0pt}}m{5cm}}
\newcommand{\hfilll}{\hspace{0pt plus 1filll}}

\graphicspath{{./fig/}}

\newtheorem{lemmas}{Lemma}
\newtheorem{theorems}{Theorem}[section]
\newtheorem{propositions}{Proposition}[section]
\newtheorem{remark}{Remark}[theorems]
\newtheorem{pro-remark}{Remark}[propositions]
\newtheorem{corollarys}{Corollary}[theorems]

\begin{document}

\title{Practical Schemes for Finding Near-Stationary Points of Convex Finite-Sums}

\author{Kaiwen Zhou\thanks{Department of Computer Science and Engineering, The Chinese University of Hong Kong, Sha Tin, N.T., Hong Kong SAR; e-mail: \href{mailto:kwzhou@cse.cuhk.edu.hk}{\tt kwzhou@cse.cuhk.edu.hk}.} \and Lai Tian\thanks{Department of Systems Engineering and Engineering Management, The Chinese University of Hong Kong, Sha Tin, N.T., Hong Kong SAR; e-mail: \href{mailto:tianlai.cs@gmail.com}{\tt tianlai.cs@gmail.com}.} \and Anthony Man-Cho So\thanks{Department of Systems Engineering and Engineering Management, The Chinese University of Hong Kong, Sha Tin, N.T., Hong Kong SAR; e-mail: \href{mailto:manchoso@se.cuhk.edu.hk}{\tt manchoso@se.cuhk.edu.hk}.} \and James Cheng\thanks{Department of Computer Science and Engineering, The Chinese University of Hong Kong, Sha Tin, N.T., Hong Kong SAR; e-mail: \href{mailto:jcheng@cse.cuhk.edu.hk}{\tt jcheng@cse.cuhk.edu.hk}.}}

\pagestyle{plain}
\date{}

\maketitle

\begin{abstract}
In convex optimization, the problem of finding near-stationary points has not been adequately studied yet, unlike other optimality measures such as the function value. Even in the deterministic case, the optimal method (OGM-G, due to \citet{OGM-G}) has just been discovered recently. In this work, we conduct a systematic study of algorithmic techniques for finding near-stationary points of convex finite-sums. Our main contributions are several algorithmic discoveries: (1) we discover a memory-saving variant of OGM-G based on the performance estimation problem approach \citep{PEP}; (2) we design a new accelerated SVRG variant that can simultaneously achieve fast rates for minimizing both the gradient norm and function value; (3) we propose an adaptively regularized accelerated SVRG variant, which does not require the knowledge of some unknown initial constants and achieves near-optimal complexities. We put an emphasis on the simplicity and practicality of the new schemes, which could facilitate future work.
\end{abstract}

\section{Introduction}
Classic convex optimization usually focuses on providing guarantees for minimizing function value. For this task, the optimal (up to constant factors) Nesterov's accelerated gradient method (NAG) \citep{AGD1,AGD2} has been known for decades, and there are even methods that can exactly match the lower complexity bounds \citep{OGM,Drori17,ITEM,drori2021oracle}. On the other hand, in general non-convex optimization, near-stationarity is the typical optimality measure, and there has been a flurry of recent research devoted to this topic \citep{ghadimi2013stochastic,ghadimi2016accelerated,Ge15,jin17a,Fang18,Zhou0G20}. Recently, there has been growing interest on devising fast schemes for finding near-stationary points in convex optimization \citep{nesterov2012make,allen2018make,foster2019complexity,carmon2021lower,KimF18,kim2018generalizing,OGM-G,ito2021nearly,diakonikolas2021potential,diakonikolas2021complementary,lee2021geometric}. This line of research is driven by the following applications and facts.
\begin{itemize}[leftmargin=15pt]
	\item \citet{nesterov2012make} studied the problem that has a linear constraint: $f(\xs)=\min_{x\in Q} {\{f(x) : Ax =b\}}$, where $Q$ is a convex set and $f$ is strongly convex. Assuming that $Q$ and $f$ are simple, we can focus on the dual problem $\phi(y^\star)=\max_y\{\phi(y) \triangleq \min_{x\in Q} {\{f(x) + \innr{y, b-Ax}\}}\}$. Clearly, the dual objective $-\phi(y)$ is smooth convex. Letting $x_y$ be the unique solution to the inner problem, we have $\nabla \phi(y) = b-Ax_y$. Note that
	$
	f(x_y) - f(\xs) = \phi(y) - \innr{y, \nabla\phi(y)} - \phi(y^\star) \leq \norm{y}\norm{\nabla \phi(y)}.
	$
	Thus, in this problem, the quantity $\norm{\nabla \phi(y)}$ serves as a measure of both primal optimality $f(x_y) \!-\! f(\xs)$ and feasibility $\norm{b\!-\!Ax_y}$, which is better than just measuring the function value.
	\item Matrix scaling \citep{rothblum1989scalings}
	is a convex problem and its goal is to find near-stationary points  \citep{allen-ZhuLOW17,cohen2017matrix}.
	\item Gradient norm is readily available, unlike other optimality measures ($f(x) - f(\xs)$ and $\norm{x - \xs}$), and is thus usable as a stopping criterion. This fact motivates the design of several parameter-free algorithms \citep{nesterov2013gradient,LinX14,ito2021nearly}, and their guarantees are established on the gradient norm.
	\item Designing schemes for minimizing the gradient norm can inspire new non-convex optimization methods. For example, SARAH \citep{SARAH} was designed for convex finite-sums with gradient-norm~measure, but was later discovered to be the near-optimal method for non-convex finite-sums~\citep{Fang18,pham2020proxsarah}.
\end{itemize}

Moreover, finding near-stationary points is often considered to be a harder task than minimizing the function value, because NAG has the optimal guarantee for $f(x) - f(\xs)$ but is only suboptimal for minimizing the gradient norm $\norm{\pf{x}}$.

\bgroup
\def\arraystretch{1.5}
\begin{table}[t]
	\caption{Finding near-stationary points $\norm{\pf{x}}\leq \epsilon$ of convex finite-sums.}
	\label{gradient_table}
	\centering
	\small
	\begin{tabular}{|c|l|c|C|}
		\hline
		&\multicolumn{1}{c|}{\textbf{Algorithm}}
		&\textbf{Complexity}\rule{0pt}{14pt}     & \textbf{Remark} \\
		\hhline{|-|-|-|-|} 
		\multirow{7}{*}{\shortstack[c]{\sf I\\\sf F\\\sf C}} &GD \hfilll\citep{OGM-G}     &$O(\frac{n}{\epsilon^2})$  & \\
		\hhline{|~|-|-|-|}
		&Regularized NAG* \hfilll\citep{carmon2021lower} &$O(\frac{n}{\epsilon}\log{\frac{1}{\epsilon}})$ &
		\\
		\hhline{|~|-|-|-|}
		&OGM-G \hfilll\citep{OGM-G}  &$O(\frac{n}{\epsilon})$  &$O(\frac{1}{\epsilon} + d)$ memory, optimal in $\epsilon$\\
		\hhline{|~|-|-|-|}
		&\cellcolor{Gray}M-OGM-G  \rule{35.7pt}{0pt}[Section \ref{sec:M-OGM-G}] &\cellcolor{Gray}$O(\frac{n}{\epsilon})$  &\cellcolor{Gray}$O(d)$ memory, optimal in $\epsilon$ \\ 
		\hhline{|~|-|-|-|}
		&L2S \hfilll\citep{li20a}  & $O(n+\frac{\sqrt{n}}{\epsilon^2})$ &Loopless variant of SARAH \citep{SARAH}\\
		\hhline{|~|-|-|-|}
		&Regularized Katyusha* \hfilll\citep{allen2018make} &$O((n+\frac{\sqrt{n}}{\epsilon})\log{\frac{1}{\epsilon}})$&Requires the knowledge of $\Delta_0$\\
		\hhline{|~|-|-|-|}
		&\cellcolor{Gray}R-Acc-SVRG-G* \rule{18.8pt}{0pt}[Section \ref{sec:R-Acc-SVRG-G}] &\cellcolor{Gray}$O((n\log{\frac{1}{\epsilon}} +\frac{\sqrt{n}}{\epsilon})\log{\frac{1}{\epsilon}})$ &\cellcolor{Gray}Without the knowledge of $\Delta_0$\\
		\hhline{|====|}
		\multirow{10}{*}{\shortstack[c]{\sf I\\\sf D\\\sf C}}
		&GD \hfilll\citep{nesterov2012make,computer-aided}   &$O(\frac{n}{\epsilon})$  &     \\
		\hhline{|~|-|-|-|}
		&NAG / NAG + GD \hfilll\citep{kim2018generalizing} / \citep{nesterov2012make}  &$O(\frac{n}{\epsilon^{2/3}})$ &      \\
		\hhline{|~|-|-|-|}
		&Regularized NAG* \hfilll\citep{nesterov2012make,ito2021nearly} & $O(\frac{n}{\sqrt{\epsilon}}\log{\frac{1}{\epsilon}})$ &  \\
		\hhline{|~|-|-|-|}
		&NAG + OGM-G \hfilll\citep{nesterov2020primal} &$O(\frac{n}{\sqrt{\epsilon}})$      & $O(\frac{1}{\sqrt{\epsilon}} + d)$ memory, optimal in $\epsilon$ \\
		\hhline{|~|-|-|-|}
		&\cellcolor{Gray}NAG + M-OGM-G \rule{1.5pt}{0pt}[Section \ref{sec:M-OGM-G}] &\cellcolor{Gray}$O(\frac{n}{\sqrt{\epsilon}})$      &\cellcolor{Gray}$O(d)$ memory, optimal in $\epsilon$\\
		\hhline{|~|-|-|-|}
		&Katyusha + L2S \hfilll[Appendix \ref{app:l2s-katyusha}] &$O(n\log{\frac{1}{\epsilon}} + \frac{\sqrt{n}}{\epsilon^{2/3}} )$ & \\
		\hhline{|~|-|-|-|}
		&\cellcolor{Gray}Acc-SVRG-G \rule{33.3pt}{0pt}[Section \ref{sec:Acc-SVRG-G}]  &\cellcolor{Gray}$O\Big(n\log{\frac{1}{\epsilon}}+ \min{\Big\{ \frac{n^{2/3}}{\epsilon^{2/3}}, \frac{\sqrt{n}}{\epsilon}\Big\}}\Big)$ &\cellcolor{Gray}$O(n\log{\frac{1}{\epsilon}}+ \sqrt{\frac{n}{\epsilon}})$ for function at the same time, simple and elegant \\
		\hhline{|~|-|-|-|}
		&Regularized Katyusha* \hfilll\citep{allen2018make} &$O((n+\sqrt{\frac{n}{\epsilon}})\log{\frac{1}{\epsilon}})$&Requires the knowledge of $R_0$\\
		\hhline{|~|-|-|-|}
		&\cellcolor{Gray}R-Acc-SVRG-G* \rule{18.8pt}{0pt}[Section \ref{sec:R-Acc-SVRG-G}] &\cellcolor{Gray}$O((n\log{\frac{1}{\epsilon}}+\sqrt{\frac{n}{\epsilon}})\log{\frac{1}{\epsilon}})$ &\cellcolor{Gray}Without the knowledge of $R_0$\\
		\hline
	\end{tabular}
	\\{\raggedleft\footnotesize{$^*$ Indirect methods (using regularization).}\par}
\end{table}
\egroup

In this work, we consider the unconstrained finite-sum problem: $\min_{x\in \R^d} {f(x) = \frac{1}{n}\sum_{i=1}^n{f_i(x)}}$, where each $f_i$ is $L$-smooth and convex. We focus on finding an $\epsilon$-stationary point of this objective function, i.e., a point with $\norm{\pf{x}}\leq \epsilon$. We use $\Xs$ to denote the set of optimal solutions, which is assumed to be nonempty. There are two different assumptions on the initial point $x_0$, namely, the Initial bounded-Function Condition (\textbf{IFC}): $f(x_0) - f(\xs) \leq \Delta_0$, and the Initial bounded-Distance Condition (\textbf{IDC}): $\norm{x_0 - \xs} \leq R_0$ for some $\xs \in \Xs$. This subtlety results in drastically different best achievable rates as studied in \citep{carmon2021lower,foster2019complexity}. Below we categorize existing techniques into three classes (relating to Table \ref{gradient_table}).

\begin{enumerate}[label=(\roman*),leftmargin=0.8cm]
	\item \textit{``IDC + IFC''.} \citet{nesterov2012make} showed that we can combine the guarantees of a method minimizing function value under IDC and a method finding near-stationary points under IFC to produce a faster one for minimizing gradient norm under IDC. For example, NAG produces $f(x_{K_1}) - f(\xs) = O(\frac{LR_0^2}{K_1^2})$ \citep{AGD1} and GD produces $\norm{\pf{x_{K_2}}}^2 = O\big(\frac{L(f(x_0) - f(\xs))}{K_2}\big)$ \citep{OGM-G} under IFC. Letting $x_0 = x_{K_1}$ and $K = K_1 + K_2$, by balancing the ratio of $K_1$ and $K_2$, we obtain the guarantee $\norm{\pf{x_K}}^2=O(\frac{L^2R_0^2}{K^3})$ for ``NAG + GD'' (same for ``NAG + OGM-G''). We point out that we can use this technique to combine the guarantees of Katyusha \citep{allen2017katyusha} and SARAH\footnote{We adopt a loopless variant of SARAH \citep{li20a}, which has a refined analysis for general convex objectives.} \citep{SARAH}; see Appendix~\ref{app:l2s-katyusha}.
	\item \textit{Regularization.} \citet{nesterov2012make} used  NAG (the strongly convex variant) to solve the regularized objective, and showed that it achieves near-optimal complexity (optimal up to log factors). Inspired by this technique, \citet{allen2018make} proposed recursive regularization for stochastic approximation algorithms, which also achieves near-optimal complexities \citep{foster2019complexity}.
	\item \textit{Direct methods.} Due to the lack of insight, existing direct methods are mostly derived or analyzed with the help of computer-aided tools \citep{KimF18,kim2018generalizing,computer-aided,OGM-G}. The computer-aided approach was pioneered by \citet{PEP}, who introduced the performance estimation problem (PEP). The only known optimal method OGM-G \citep{OGM-G} was designed based on the PEP approach.
\end{enumerate}

Observe that since $f(x) - f(\xs) \leq \norm{\pf{x}}\norm{x - \xs}$, the lower bound for finding near-stationary points must be of the same order as for minimizing function value \citep{AGD3}. Thus, under IDC, the lower bound is $\Omega(n+\sqrt{\frac{n}{\epsilon}})$ due to \citep{woodworthS16}. Under IFC, we can establish an $\Omega(n+\frac{\sqrt{n}}{\epsilon})$ lower bound using the techniques in \citep{carmon2021lower,woodworthS16}. 
The main contributions of this work are three new algorithmic schemes that improve the practicalities of existing methods, which is summarized below (highlighted in~Table~\ref{gradient_table}).
\begin{itemize}[leftmargin=15pt]
	\item (Section \ref{sec:OGM-G}) We propose a memory-saving variant of OGM-G for the deterministic case ($n=1$), which does not require pre-computed and stored parameters. The derivation of the new variant is inspired by the numerical solution to a PEP problem. 
	\item (Section \ref{sec:Acc-SVRG-G}) We propose a new accelerated SVRG \citep{SVRG,Prox_SVRG} variant that can \textit{simultaneously} achieve fast rates for minimizing both the gradient norm and function value, that is, $O(n\log{\frac{1}{\epsilon}}+ \min{\{ \frac{n^{2/3}}{\epsilon^{2/3}}, \frac{\sqrt{n}}{\epsilon}\}})$ complexity for gradient norm and $O(n\log{\frac{1}{\epsilon}}+ \sqrt{\frac{n}{\epsilon}})$ complexity for function value. Other stochastic approaches in Table \ref{gradient_table} do not have this property. 
	\item (Section \ref{sec:R-Acc-SVRG-G}) We propose an adaptively regularized accelerated SVRG variant, which does not require the knowledge of $R_0$ or $\Delta_0$ and achieves a near-optimal complexity under IDC or IFC. 
\end{itemize}
We put in extra efforts to make the proposed schemes as simple and elegant as possible. We believe that the simplicity makes extensions of the new schemes easier.

\section{Preliminaries}

Throughout this paper, we use $\innr{\cdot,\cdot}$ and $\norm{\cdot}$ to denote the inner product and the Euclidean norm, respectively. We let $[n]$ denote the set $\{1,2,\ldots,n\}$, $\mathbb{E}$ denote the total expectation and $\mathbb{E}_{i_k}$ denote the expectation with respect to a random sample $i_k$. We say that a function $f: \R^d \rightarrow \R$ is \textit{$L$-smooth} if it has $L$-Lipschitz continuous gradients, i.e.,
$
\forall x, y\in \R^d, \norm{\pf{x} - \pf{y}} \leq L\norm{x - y}.
$
A continuously differentiable $f$ is called \textit{$\mu$-strongly convex} if
$
\forall x, y\in \R^d, f(x) - f(y) - \innr{\pf{y}, x - y} \geq \frac{\mu}{2} \norm{x - y}^2.
$
Other equivalent definitions of these two assumptions can be found in the textbook \citep{AGD3}. The following is an important consequence of a function $f$ being $L$-smooth and convex. 
\begin{equation}\label{interpolation_c}
\forall x, y\in \R^d,
  f(x) - f(y) - \innr{\pf{y}, x - y}\geq \frac{1}{2L} \norm{\pf{x} - \pf{y}}^2.
\end{equation}
We call it \textit{interpolation condition} at $(x,y)$ following \citep{taylor2017smooth}. If $f$ is both $L$-smooth and $\mu$-strongly convex, we can define a ``shifted'' function $h(x) = f(x) - f(\xs) - \frac{\mu}{2} \norm{x - \xs}^2$ following \citep{zhou2020boosting}. It can be easily verified that $h$ is $(L-\mu)$-smooth and convex, and thus from \eqref{interpolation_c},
\begin{equation}\label{interpolation_sc}
	\forall x, y\in \R^d,
	h(x) - h(y) - \innr{\ph{y}, x - y}\geq \frac{1}{2(L-\mu)} \norm{\ph{x} - \ph{y}}^2,
\end{equation}
which is equivalent to the \textit{strongly convex interpolation condition} discovered in \citep{taylor2017smooth}.

Oracle complexity (or simply complexity) refers to the required number of stochastic gradient $\nabla f_i$ computations to find an $\epsilon$-accurate solution.

\section{OGM-G: ``Momentum'' Reformulation and a Memory-Saving Variant}
\label{sec:OGM-G}

In this section, we focus on the IFC setting, that is, $f(x_0) - f(\xs) \leq \Delta_0$. We use $N$ to denote the total number of iterations (each computes a full gradient $\nabla f$). Proofs in this section are given in Appendix~\ref{app:proof-ogm}. Recall that OGM-G has the following updates~\citep{OGM-G}. Let $y_0 = x_0$. For $k=0,\ldots, N-1,$
\begin{equation}\label{alg:o-OGM-G}
\begin{aligned}
	y_{k+1} &= x_k - \frac{1}{L} \pf{x_k},\\
	x_{k+1} &= y_{k+1} + \frac{(\theta_k - 1)(2\theta_{k+1} - 1)}{\theta_{k}(2\theta_k - 1)}(y_{k+1} - y_k) + \frac{2\theta_{k+1} - 1}{2\theta_{k} - 1}(y_{k+1} - x_k),
\end{aligned} 
\end{equation}
where the sequence $\{\theta_k\}$ is recursively defined: $\theta_N = 1$ and 
$
\begin{cases}
	\theta_k^2 - \theta_k = \theta_{k+1}^2 & k = 1\ldots N-1, \\
	\theta_0^2 - \theta_0 = 2\theta_1^2 &\text{otherwise.}
\end{cases}
$ 

OGM-G was discovered from the numerical solution to an SDP problem and its analysis is to show that the step coefficients in \eqref{alg:o-OGM-G} specify a feasible solution to the SDP problem. While this analysis is natural for the PEP approach, it is hard to understand how each coefficient affects the rate, especially if one wants to generalize the scheme. Here we provide a simple algebraic analysis for OGM-G. 

We start with a reformulation\footnote{It can be verified that this scheme is equivalent to the original one \eqref{alg:o-OGM-G} through $v_{k} = \frac{1}{(2\theta_k - 1)\theta_k^2} (y_k - x_k)$.} of OGM-G in Algorithm \ref{alg:OGM-G}, which aims to simplify the proof. We adopt a consistent sequence $\{\theta_k\}$: $\theta_N = 1$ and $\theta_k^2 - \theta_k = \theta_{k+1}^2$, $k = 0\ldots N-1$, which only costs a constant factor.\footnote{The original guarantee of OGM-G can be recovered if we set $\theta_0^2 - \theta_0 = 2\theta_1^2$.} Interestingly, the reformulated scheme resembles the heavy-ball momentum method \citep{polyak1964some}. However, it can be shown that Algorithm \ref{alg:OGM-G} is not covered by the heavy-ball momentum scheme. Defining $\theta_{N+1}^2 = \theta_{N}^2 - \theta_{N} = 0$, we provide the one-iteration analysis in the following proposition:

\begin{algorithm}[t]
	\caption{OGM-G: ``Momentum'' reformulation}
	\label{alg:OGM-G}
	\renewcommand{\algorithmicrequire}{\textbf{Input:}}
	\renewcommand{\algorithmicensure}{\textbf{Initialize:}}
	\begin{algorithmic}[1]
		\Require initial guess $x_0 \in \R^d$, total iteration number $N$.
		\Ensure vector $v_0 = \mathbf{0}$, scalars $\theta_N = 1$ and $\theta_k^2 - \theta_k = \theta_{k+1}^2$, for $k = 0\ldots N-1$.
		\For{$k=0, \ldots, N-1$}
		\State $v_{k+1} = v_k + \frac{1}{L\theta_k\theta_{k+1}^2}\pf{x_k}$.
		\State $x_{k+1} = x_k - \frac{1}{L}\pf{x_k} - (2\theta_{k+1}^3 - \theta_{k+1}^2) v_{k+1}$.
		\EndFor
		\renewcommand{\algorithmicensure}{\textbf{Output:}}
		\Ensure $x_N$.
	\end{algorithmic}
\end{algorithm}

\begin{propositions}\label{prop:OGM-G}
	In Algorithm \ref{alg:OGM-G}, the following holds at any iteration $k\in\{0,\ldots, N-1\}$. 
	\begin{equation}\label{T1}
		\begin{aligned}
			A_k + B_{k+1} + C_{k+1} + E_{k+1} \leq{}& A_{k+1} + B_k + C_k + E_k- \theta_{k+1}\innr{\pf{x_{k+1}},  v_{k+1}} \\
			&+ \sum_{i=k+1}^{N}{\frac{\theta_i}{L\theta_k\theta_{k+1}^2}\innr{\pf{x_k},\pf{x_i}}},
		\end{aligned} 
	\end{equation}
	where $A_k\triangleq \frac{1}{\theta_k^2} (f(x_N) - f(\xs) - \frac{1}{2L}\norm{\pf{x_N}}^2)$, $B_k \triangleq \frac{1}{\theta_k^2} (f(x_k) - f(\xs))$, $C_k \triangleq \frac{1}{2L\theta_k^2}\norm{\pf{x_k}}^2$ and $E_k \triangleq \frac{\theta_{k+1}^2}{\theta_k}\innr{\pf{x_k}, v_k}$.
\end{propositions}

\begin{pro-remark}
A recent work \citep{diakonikolas2021potential} also conducted an algebraic analysis of OGM-G under a potential function framework. Their potential function decrease can be directly obtained from Proposition~\ref{prop:OGM-G} by summing up \eqref{T1}. By contrast, our ``momentum'' vector $\{v_k\}$  naturally merges into the analysis, which significantly simplifies the analysis. Moreover, it provides a better interpretation on how OGM-G utilizes the past gradients to achieve acceleration. A concurrent work \citep{lee2021geometric} discovered the potential function of OGM-G while their analysis is much more complicated.
\end{pro-remark}

From \eqref{T1}, we see that only the last two terms do not telescope. Note that the ``momentum'' vector is a weighted sum of the past gradients, i.e., $v_{k+1} = \sum_{i=0}^{k}{\frac{1}{L\theta_i\theta_{i+1}^2}\pf{x_i}}$. If we sum the terms up from $k=0,\ldots, N-1$, it can be verified that they exactly sum up to $0$. Then, by telescoping the remaining terms, we obtain the final convergence guarantee.

\begin{theorems}\label{thm:OGM-G} The output of Algorithm \ref{alg:OGM-G} satisfies
	$
	\norm{\pf{x_N}}^2\leq \frac{8L \Delta_0}{(N+2)^2}.
	$
\end{theorems}

We observe two drawbacks of OGM-G (which have been similarly pointed out in \citep{diakonikolas2021potential,lee2021geometric}): (1) it requires storing a pre-computed parameter sequence, which costs $O(\frac{1}{\epsilon})$ floats; (2)  except for the last iterate, all other iterates do not have properly upper-bounded gradient norms. 
We resolve these issues by proposing another parameterization of Algorithm \ref{alg:OGM-G} in the next subsection.

\subsection{Memory-Saving OGM-G}
\label{sec:M-OGM-G}

\begin{algorithm}[t]
	\caption{M-OGM-G: Memory-saving OGM-G}
	\label{alg:M-OGM-G}
	\renewcommand{\algorithmicrequire}{\textbf{Input:}}
	\renewcommand{\algorithmicensure}{\textbf{Initialize:}}
	\begin{algorithmic}[1]
		\Require initial guess $x_0 \in \R^d$, total iteration number $N$.
		\Ensure vector $v_0 = \mathbf{0}$.
		\For{$k=0, \ldots, N-1$}
		\State $v_{k+1} = v_k + \frac{12}{L(N-k+1)(N-k+2)(N-k+3)}\pf{x_k}$.
		\State $x_{k+1} = x_k - \frac{1}{L}\pf{x_k} - \frac{(N-k)(N-k+1)(N-k+2)}{6} v_{k+1}$.
		\EndFor
		\renewcommand{\algorithmicensure}{\textbf{Output:}}
		\Ensure $x_N$ or $\argmin_{x\in \{x_0,\ldots, x_N\}} {\norm{\pf{x}}}$.
	\end{algorithmic}
\end{algorithm}

A straightforward idea to resolve the aforementioned issues is to generalize Algorithm \ref{alg:OGM-G}. However, we find it rather difficult since the parameters in the analysis are rather strict (despite that the proof is already simple). We choose to rely on computer-aided techniques \citep{PEP}. The derivation of this variant (Algorithm \ref{alg:M-OGM-G}) is based on the following numerical~experiment.

\paragraph{Numerical experiment.}\!\!\! OGM-G was discovered when considering the relaxed PEP problem \citep{OGM-G}:
\begin{equation}
	\label{PEP-P}\tag{P}
	\begin{aligned}
		&\max_{\substack{\pf{x_0},\ldots, \pf{x_N} \in \R^d \\ f(x_0),\ldots, f(x_N),f(\xs)\in \R}} {\norm{\pf{x_N}}^2} \\
		\text{subject to} &\begin{cases}
			\text{interpolation condition \eqref{interpolation_c} at }(x_k, x_{k+1}), \ \ k=0,\ldots, N-1,\\
			\text{interpolation condition \eqref{interpolation_c} at }(x_N, x_k), \quad\,\, k=0,\ldots, N-1,\\
			\text{interpolation condition \eqref{interpolation_c} at }(x_N, \xs),  \ \ f(x_0) - f(\xs) \leq \Delta_0,
		\end{cases}
	\end{aligned}
\end{equation}
where the sequence $\{x_k\}$ is defined as $x_{k+1} = x_k - \frac{1}{L}\sum_{i=0}^k{h_{k+1,i}\pf{x_i}}, k = 0,\ldots, N-1$ for some step coefficients $h\in \R^{N(N+1)/2}$. Given $N$, the step coefficients of OGM-G correspond to a numerical solution to the problem: $\argmin_{h} \{\text{Lagrangian dual of \eqref{PEP-P}}\}$, which is denoted as (HD). Conceptually, solving problem (HD) would give us the fastest possible step coefficients under the constraints.\footnote{However, since problem (HD) is non-convex, we can only obtain local solutions.} We expect there to be some constant-time slower schemes, which are neglected when solving (HD). To identify them, we relax a set of interpolation conditions in problem \eqref{PEP-P}:
\[
f(x_N) - f(x_k) - \innr{\pf{x_k}, x_N - x_k} \geq \frac{1}{2L} \norm{  \pf{x_N} - \pf{x_k}}^2 - \rho \norm{\pf{x_k}}^2,
\]
for $k=0,\ldots, N-1$ and some $\rho>0$. After this relaxation, solving (HD) will no longer give us the step coefficients of OGM-G. Moreover, the subtracted term $\rho\norm{\pf{x_k}}^2$ forces the PEP tool to not ``utilize'' it (to cancel out other terms) when searching for step coefficients. Since such a term is not ``utilized'' in each of the $N$ interpolation conditions, after summation, these terms appear on the left hand side of \eqref{M-OGM-G-rate}, which gives upper bounds to the gradient norms evaluated at intermediate iterates. By trying different $\rho$ and checking the dependence on $N$, we discover Algorithm \ref{alg:M-OGM-G} when $\rho = \frac{1}{2L}$. Similar to our analysis of OGM-G, we provide a simple algebraic analysis in the following theorem.
\begin{theorems}\label{thm:M-OGM-G}
	Define $\delta_{k+1} \!\triangleq\! \frac{12}{(N-k+1)(N-k+2)(N-k+3)}, k=0,\ldots, N$. In Algorithm \ref{alg:M-OGM-G}, it holds that
	\begin{equation}\label{M-OGM-G-rate}
	\sum_{k=0}^{N}{\frac{\delta_{k+1}}{2}\norm{\pf{x_k}}^2} \leq \frac{12L \Delta_0}{(N+2)(N+3)}.
	\end{equation}
\end{theorems}

\begin{remark}
	From \eqref{M-OGM-G-rate}, we can directly conclude that $\forall k \in \{0, \ldots, N\}, \norm{\pf{x_k}}^2 = O(\frac{L\Delta_0}{N^2 \delta_{k+1}})$ and thus, the rate (in terms of $N$) on the last iterate is optimal (since $\delta_{N+1} = 2$). Moreover, the minimum gradient also achieves the optimal rate since
	\[
		\min_{k\in\{0,\ldots,N\}} {\norm{\pf{x_k}}^2} 
		\leq{} \frac{1}{\sum_{k=0}^N{\frac{\delta_{k+1}}{2}}} \sum_{k=0}^N{\frac{\delta_{k+1}}{2}\norm{\pf{x_k}}^2} \leq{} \frac{8L \Delta_0}{(N+2)(N+3) - 2}.
	\] 
\end{remark}

Clearly, the parameters of this variant can be computed on the fly and from the above remark, each iterate has an upper-bounded gradient norm. The constructions in \citep{diakonikolas2021potential,lee2021geometric} all require pre-computed and stored sequences, which seems to be unavoidable in their analysis as admitted in \citep{diakonikolas2021potential}. Our discovery is another example of the powerfulness of computer-aided methodology, which finds proofs that are difficult or even impossible to find with bare hands. 
We can extend the benefits into the IDC setting using the ideas in \citep{nesterov2012make} as summarized below.

\begin{corollarys}[IDC case]\label{coroll:NAG+M-OGM-G}
	If we first run $N/2$ iterations of NAG and then continue with $N/2$ iterations of Algorithm \ref{alg:M-OGM-G}, we obtain an output satisfying $\norm{\pf{x_N}}^2 = O\big(\frac{L^2R_0^2}{N^4}\big)$. 
\end{corollarys}

\section{Accelerated SVRG: Fast Rates for Both Gradient Norm and Objective}
\label{sec:Acc-SVRG-G}
In this section, we focus on the IDC setting, that is, $\norm{x_0 - \xs} \leq R_0$ for some $\xs \in \Xs$. We use $K$ to denote the total number of stochastic iterations. From the development in Section \ref{sec:OGM-G}, it is natural to ask whether we can use the PEP approach to motivate new stochastic schemes. However, due to the exponential growth of the number of possible states $(i_0, i_1,\ldots)$, we cannot directly adopt this approach. A feasible alternative is to first fix an algorithmic framework and a family of potential functions, and then use the potential-based PEP approach in \citep{computer-aided}. However, this approach is much more restrictive. For example, it cannot identify special constructions like \eqref{T1} in OGM-G. Fortunately, as we will see, we can get some inspiration from the recent development of deterministic methods. Proofs in this section are given in Appendix~\ref{app:proof-acc-svrg-g}.

Our proposed scheme is given in Algorithm \ref{alg:Acc-SVRG-G}. We adopt the elegant loopless design of SVRG in \citep{kovalev2020don}. Note that the full gradient $\pf{\tilde{x}_k}$ is computed and stored only when $\tilde{x}_{k+1} = y_k$ at Step \ref{alg-step:prob-step}.  We summarize our main technical novelty as follows.

\begin{algorithm}[t]
	\caption{Acc-SVRG-G: Accelerated SVRG for Gradient minimization}
	\label{alg:Acc-SVRG-G}
	\renewcommand{\algorithmicrequire}{\textbf{Input:}}
	\renewcommand{\algorithmicensure}{\textbf{Initialize:}}
	\begin{algorithmic}[1]
		\Require parameters $\{\tau_k\}$, $\{p_k\}$, initial guess $x_0 \in \R^d$, total iteration number $K$.
		\Ensure vectors $z_0 = \tilde{x}_0 = x_0$  and scalars $\alpha_k = \frac{L\tau_k}{1 - \tau_k}, \forall k$ and $\widetilde{\tau} = \sum_{k=0}^{K-1}{\tau_k^{-2}}$.
		\For{$k=0, \ldots, K-1$}
		\State $y_{k} = \tau_k z_k + \left(1 - \tau_k\right) \left(\tilde{x}_k - \frac{1}{L}\pf{\tilde{x}_k}\right).$
		\State $z_{k+1} = \arg\min_{x} \left\{\innr{\mathcal{G}_k, x} + (\alpha_k / 2) \norm{x - z_k}^2\right\}$.\label{alg-step:mirror_c}
		\State \slash\slash\,$\mathcal{G}_k \triangleq \pfik{y_k} - \pfik{\tilde{x}_k} + \pf{\tilde{x}_k},$
		where $i_k$ is sampled uniformly in $[n]$.
		\State\label{alg-step:prob-step} $\tilde{x}_{k+1} = \begin{cases}
			y_k &\text{with probability }p_k,\\
			\tilde{x}_k &\text{with probability }1 - p_k. 
		\end{cases}$
		\EndFor
		\renewcommand{\algorithmicensure}{\textbf{Output (for gradient):}}
		\Ensure $x_{\text{out}}$ is sampled from $\left\{ \text{Prob}\{x_{\text{out}} = \tilde{x}_k\} = \frac{\tau_k^{-2}}{\widetilde{\tau}} \cond k\in \{0, \ldots, K-1\} \right\}$.
		\renewcommand{\algorithmicensure}{\textbf{Output (for function value):}}
		\Ensure $\tilde{x}_K$.
	\end{algorithmic}
\end{algorithm}

\paragraph{Main algorithmic novelty.} The design of stochastic accelerated methods is largely inspired by NAG. To make it clear, by setting $n=1$, we see that Katyusha \citep{allen2017katyusha}, MiG \citep{MiG}, SSNM \citep{SSNM}, Varag~\citep{VARAG}, VRADA \citep{VRADA}, ANITA \citep{Anita}, the acceleration framework in \citep{acc-vr} and AC-SA \citep{lan2012optimal,ghadimi2012optimal,pmlr-v124-zhou20a} all reduce to one of the following variants of NAG \citep{AT-NAG,zhu:linear}. We say that these methods are under the NAG framework.
\[
\begin{aligned}
&\begin{cases}
	x_k= \tau_k z_k + (1 - \tau_k) y_k,\\
	z_{k+1} = z_k - \alpha_k \pf{x_k},\\
	y_{k+1} = \tau_k z_{k+1} + (1 - \tau_k) y_k.\\	
\end{cases} &&&&&&\begin{cases}
x_k= \tau_k z_k + (1 - \tau_k) y_k,\\
z_{k+1} = z_k - \alpha_k \pf{x_k},\\
y_{k+1} = x_k - \eta_k \pf{x_k}.\\	
\end{cases}\\
&\ \ \ \ \text{\citet{AT-NAG}} &&&&&&\ \ \ \ \ \ \text{Linear Coupling \citep{zhu:linear}} 
\end{aligned}
\]
See \citep{u_AGD, NAGs} for other variants of NAG. When $n=1$, Algorithm \ref{alg:Acc-SVRG-G} reduces to the following scheme.
\[
\begin{aligned}
	&\begin{cases}
	y_{k} = \tau_k z_k + \left(1 - \tau_k\right) \left(y_{k-1} - \frac{1}{L}\pf{y_{k-1}}\right),\\
	z_{k+1} = z_k -  \frac{1}{\alpha_k} \pf{y_k}.
	\end{cases}\\
	&\ \ \ \ \text{Optimized Gradient Method (OGM) \citep{PEP,OGM}}
\end{aligned}
\]
Algorithm \ref{alg:Acc-SVRG-G} reduces to the scheme of OGM when $n=1$ (this point is clearer in the formulation of ITEM in \citep{ITEM}). Note that although we use OGM as the inspiration, the original OGM has nothing to do with making the gradient small and there is no hint on how a stochastic variant can be designed. OGM has a constant-time faster worst-case rate than NAG, which exactly matches the lower complexity bound in \citep{Drori17}. In the following proposition, we show that the OGM framework helps us conduct a tight one-iteration analysis, which gives room for achieving our goal.

\begin{propositions}\label{prop:Acc-SVRG-final}
	In Algorithm \ref{alg:Acc-SVRG-G}, the following holds at any iteration $k\geq0$ and $\forall \xs \in \Xs$.
	\begin{equation}\label{P10}
	\begin{aligned}
		&\left(\frac{1-\tau_k}{\tau_k^2p_k}\E{f(\tilde{x}_{k+1}) - f(\xs)} + \frac{L}{2} \E{\norm{z_{k+1} - \xs}^2}\right) + \frac{(1 -\tau_k)^2}{2L\tau_k^2}\E{\norm{\pf{\tilde{x}_k}}^2} \\ \leq{}& \left(\frac{(1-\tau_kp_k)(1-\tau_k)}{\tau_k^2p_k}\E{f(\tilde{x}_k) - f(\xs)} + \frac{L}{2} \E{\norm{z_k - \xs}^2}\right). 
	\end{aligned}
	\end{equation}
\end{propositions}

The terms inside the parentheses form the commonly used potential function of SVRG variants. The additional $\mathbb{E}{[\norm{\pf{\tilde{x}_k}}^2]}$ term is created by adopting the OGM framework. In other words, we use the following potential function for Algorithm \ref{alg:Acc-SVRG-G} ($a_k,b_k, c_k \geq 0$):
\[
	T_k = a_k\E{f(\tilde{x}_k) - f(\xs)} + b_k \E{\norm{z_k - \xs}^2} + \sum_{i = 0}^{k-1}{c_i\E{\norm{\pf{\tilde{x}_i}}^2}}.
\] 
We first provide a simple parameter choice, which leads to a simple and clean analysis.
\begin{theorems}[Single-stage parameter choice]\label{thm:Acc-SVRG-G-single-stage}
	In Algorithm \ref{alg:Acc-SVRG-G}, if we choose $p_k \equiv \frac{1}{n}, \tau_k = \frac{3}{k/n + 6}$, the following holds at the outputs.
	\begin{equation}\label{P21}
	\begin{gathered}
		\E{\norm{\pf{x_{\textup{out}}}}^2} 
		= O\left(\frac{n^3L\big(f(x_0) - f(\xs)\big) + n^2L^2 R_0^2}{K^3}\right),\\
		\E{f(\tilde{x}_K)} - f(\xs) = O\left(\frac{n^2\big(f(x_0) - f(\xs)\big) + nLR_0^2}{K^2}\right).
	\end{gathered}
	\end{equation}
	In other words, to guarantee $\E{\norm{\pf{x_{\textup{out}}}}}\leq \epsilon_g$ and $\E{f(\tilde{x}_K)} - f(\xs) \leq \epsilon_f$, the oracle complexities are $O\Big(\frac{n(L(f(x_0) - f(\xs)))^{1/3}}{\epsilon_g^{2/3}} + \frac{(nLR_0)^{2/3}}{\epsilon_g^{2/3}}\Big)$ and 
	$O\Big(n\sqrt{\frac{f(x_0) - f(\xs)}{\epsilon_f}} + \frac{\sqrt{nL}R_0}{\sqrt{\epsilon_f}}\Big)$, respectively.
\end{theorems}
From \eqref{P21}, we see that Algorithm \ref{alg:Acc-SVRG-G} achieves fast $O(\frac{1}{K^{1.5}})$ and $O(\frac{1}{K^2})$ rates for minimizing the gradient norm and function value at the same time. However, despite being a simple choice, the oracle complexities are not better than the deterministic methods in Table \ref{gradient_table}. Below we provide a two-stage parameter choice, which is inspired by the idea of including a ``warm-up phase'' in \citep{SVRGpp,VARAG,VRADA,Anita}.

\begin{theorems}[Two-stage parameter choice]\label{thm:Acc-SVRG-G-multiple-stage}
	In Algorithm \ref{alg:Acc-SVRG-G}, let $p_k = \max\{\frac{6}{k+8}, \frac{1}{n}\}, \tau_k = \frac{3}{p_k(k+8)}$. The oracle complexities needed to guarantee $\E{\norm{\pf{x_{\textup{out}}}}}\leq \epsilon_g$ and $\E{f(\tilde{x}_K)} - f(\xs) \leq \epsilon_f$ are 
	\[
		O\left(n\min{\left\{\log{\frac{LR_0}{\epsilon_g}}, \log{n}\right\}}+ \frac{(nLR_0)^{2/3}}{\epsilon_g^{2/3}}\right) \text{ and }\, O\left(n\min{\left\{\log{\frac{LR_0^2}{\epsilon_f}}, \log{n}\right\}}+ \frac{\sqrt{nL}R_0}{\sqrt{\epsilon_f}}\right),
	\]
	respectively.
\end{theorems}
Since $\norm{\pf{\tilde{x}_K}}^2 = O\big(L\big(f(\tilde{x}_K) - f(\xs)\big)\big)$, the last iterate has the complexity $O(n\log{\frac{1}{\epsilon}}+ \frac{\sqrt{n}}{\epsilon})$ for minimizing the gradient norm. Then, by outputting the $\tilde{x}$ that attains the minimum gradient, we can combine the results of outputting $x_\text{out}$ and $\tilde{x}_K$, which leads to the complexity $O(n\log{\frac{1}{\epsilon}}+ \min{\{ \frac{n^{2/3}}{\epsilon^{2/3}}, \frac{\sqrt{n}}{\epsilon}\}})$ in Table~\ref{gradient_table}. This complexity has a slightly worse dependence on $n$ than Katyusha + L2S. It is due to the adoption of $n$-dependent step size in L2S. As studied in \citep{li20a}, despite having a better complexity, $n$-dependent step size boosts numerical performance only when $n$ is \textit{extremely large}. If the practically fast $n$-independent step size is used for L2S, Katyusha+L2S and Acc-SVRG-G have similar complexities. See also Appendix \ref{app:exp}.

If $\epsilon$ is large or $n$ is very large, the recently proposed ANITA \citep{Anita} achieves an $O(n)$ complexity, which matches the lower complexity bound $\Omega(n)$ in this case \citep{woodworthS16}. Since ANITA uses the NAG framework, we show that similar results can be derived under the OGM framework in the following theorem:

\begin{theorems}[Low accuracy parameter choice]
	\label{thm:Acc-SVRG-G-first-epoch}
	In Algorithm \ref{alg:Acc-SVRG-G}, let iteration $N$ be the first time Step~\ref{alg-step:prob-step} updates $\tilde{x}_{k+1} = y_k$. If we choose $p_k \equiv \frac{1}{n}$, $\tau_k \equiv 1 - \frac{1}{\sqrt{n+1}}$ and terminate Algorithm \ref{alg:Acc-SVRG-G} at iteration $N$, then the following holds at $\tilde{x}_{N+1}:$
	\[
		\E{\norm{\pf{\tilde{x}_{N+1}}}^2}\leq \frac{8L^2R_0^2}{5(\sqrt{n+1}+1)}\text{ and }\, \E{f(\tilde{x}_{N+1})} - f(\xs) \leq \frac{LR_0^2}{\sqrt{n+1}+1},
	\]
	In particular, if the required accuracies are low (or $n$ is very large), i.e., $\epsilon_g^2 \geq \frac{8L^2R_0^2}{5(\sqrt{n+1}+1)}$ and $\epsilon_f \geq \frac{LR_0^2}{\sqrt{n+1}+1}$, then Algorithm \ref{alg:Acc-SVRG-G} only has an $O(n)$ oracle complexity.
\end{theorems}

In the low accuracy region (specified above), the choice in Theorem \ref{thm:Acc-SVRG-G-first-epoch} removes the $O(\log{\frac{1}{\epsilon}})$ factor in the complexity of Theorem \ref{thm:Acc-SVRG-G-multiple-stage}. From the above two theorems, we see that Algorithm \ref{alg:Acc-SVRG-G} achieves a similar rate for minimizing the function value as ANITA \citep{Anita}, which is the current best rate. We include some numerical justifications of Algorithm \ref{alg:Acc-SVRG-G} in Appendix~\ref{app:exp}. We believe that the potential-based PEP approach in \citep{computer-aided} can help us identify better parameter choices of Algorithm \ref{alg:Acc-SVRG-G}, which we leave for future work.

\section{Near-Optimal Accelerated SVRG with Adaptive Regularization}
\label{sec:R-Acc-SVRG-G}

\begin{algorithm}[t]
	\caption{R-Acc-SVRG-G}
	\label{alg:R-Acc-SVRG-G}
	\renewcommand{\algorithmicrequire}{\textbf{Input:}}
	
	\begin{algorithmic}[1]
		\Require accuracy $\epsilon>0$, parameters $\delta_0 = L, \beta > 1$, initial guess $x_0 \in \R^d$.
		\For{$t=0,1,2, \ldots$}
		
		\State\label{alg-step:regularization} Define $f^{\delta_t}(x) = (1/n)\sum_{i=1}^n{f_i^{\delta_t}(x)}$, where $f^{\delta_t}_i(x) = f_i(x) + (\delta_t/2) \norm{x - x_0}^2$.
		\State Initialize vectors $z_0 = \tilde{x}_0 = x_0$ and set $\tau_x, \tau_z, \alpha, p, C_{\textup{IDC}}, C_{\textup{IFC}}$ according to Proposition \ref{prop:R-params}.
		\For{$k=0,1,2,\ldots$} \label{alg-step:inner-loop}
		\State $y_{k} = \tau_x z_k + \left(1 - \tau_x\right) \tilde{x}_k +  \tau_{z}\left(\delta_t(\tilde{x}_k - z_k) - \pfd{\tilde{x}_k}\right)$.
		\State $z_{k+1} = \arg\min_{x} \Big\{\innr{\mathcal{G}^{\delta_t}_k, x} + (\alpha / 2) \norm{x - z_k}^2 + (\delta_t/2) \norm{x - y_k}^2 \Big\}$.
		\State \slash\slash\,$\mathcal{G}^{\delta_t}_k \triangleq \pfdik{y_k} - \pfdik{\tilde{x}_k} + \pfd{\tilde{x}_k},$
		where $i_k$ is sampled uniformly in $[n]$.
		\State $\tilde{x}_{k+1} = \begin{cases}
			y_k &\text{with probability }p,\\
			\tilde{x}_k &\text{with probability }1 - p. 
		\end{cases}$
		\If {\footnotemark$\norm{\pf{\tilde{x}_k}}\leq \epsilon$} output $\tilde{x}_k$ and terminate the algorithm. \label{alg-step:terminate}
		\EndIf
		\If {under IDC and $ (1 + \frac{\delta_t}{\alpha})^k \geq \sqrt{C_\textup{IDC}} / \delta_t$} break the inner loop. \label{alg-step:IDC-break}
		\EndIf
		\If {under IFC and $ (1 + \frac{\delta_t}{\alpha})^k \geq \sqrt{C_\textup{IFC} / 2\delta_t}$} break the inner loop. \label{alg-step:IFC-break}
		\EndIf
		\EndFor \label{alg-step:inner-loop-end}
		\State $\delta_{t+1} = \delta_t / \beta$.
		\EndFor
	\end{algorithmic}
\end{algorithm}
\footnotetext{Note that we maintain the full gradient $\pfd{\tilde{x}_k}$ and $\pf{\tilde{x}_k} = \pfd{\tilde{x}_k} - \delta_t(\tilde{x}_k - x_0)$.}

Currently, there is no known stochastic method that directly achieves the optimal rate in~$\epsilon$. To get near-optimal rates, the existing strategy is to use a carefully designed regularization technique \citep{nesterov2012make,allen2018make} with a method that solves strongly convex problems; see, e.g., \citep{nesterov2012make,allen2018make,foster2019complexity, davis2018complexity}. However, the regularization parameter requires the knowledge of $R_0$ or $\Delta_0$, which significantly limits its practicality.

Inspired by the recently proposed adaptive regularization technique \citep{ito2021nearly}, we develop a near-optimal accelerated SVRG variant (Algorithm \ref{alg:R-Acc-SVRG-G}) that does not require the knowledge of $R_0$ or $\Delta_0$. Note that this technique was originally proposed for NAG under the IDC assumption. Our development extends this technique to the stochastic setting, which brings an $O(\sqrt{n})$ rate improvement compared with adaptive regularized NAG. Moreover, we consider both IFC and IDC settings.  Proofs in this section are in Appendix~\ref{app:proof-r-acc-svrg-g}.

\paragraph{Detailed design.} Algorithm \ref{alg:R-Acc-SVRG-G} has a ``guess-and-check'' framework. In the outer loop, we first define the regularized objective $f^{\delta_t}$ using the current estimate of regularization parameter $\delta_t$, and then we initialize an accelerated SVRG method (the inner loop) to solve the $\delta_t$-strongly convex $f^{\delta_t}$. If the inner loop breaks at Step \ref{alg-step:IDC-break} or \ref{alg-step:IFC-break}, indicating the poor quality\footnote{If Algorithm \ref{alg:R-Acc-SVRG-G} does not terminate before it breaks at Step \ref{alg-step:IDC-break} or \ref{alg-step:IFC-break} for the current estimate $\delta_t$, it is quite likely that running infinite number of inner iterations, the algorithm still will not terminate.} of the current estimate $\delta_t$, $\delta_t$ will be divided by a fixed $\beta$. Thus, conceptually, we can adopt any method that solves strongly convex finite-sums at the optimal rate as the inner loop. However, since the constructions of Step~\ref{alg-step:IDC-break} or \ref{alg-step:IFC-break} require some algorithm-dependent constants, we have to fix one method as the inner loop. 

The inner loop we adopted is a loopless variant of BS-SVRG \citep{zhou2020boosting}. This is because (i) BS-SVRG is the fastest known accelerated SVRG variant (for ill-conditioned problems) and (ii) it has a simple scheme, especially after using the loopless construction \citep{kovalev2020don}. However, its original guarantee is built upon $\{z_k\}$. Clearly, we cannot implement the stopping criterion (Step~\ref{alg-step:terminate}) on $\norm{\pf{z_k}}$. Interestingly, we discover that its sequence $\{\tilde{x}_k\}$ works perfectly in our regularization framework, even if we can neither establish convergence on $f(\tilde{x}_k) - f(\xs)$ nor on $\norm{\tilde{x}_k - \xs}^2$.\footnote{It is due to the special potential function of BS-SVRG (see \eqref{app:bs-svrg-potential}), which does not contain these two terms.} Moreover, we find that the loopless construction significantly simplifies the parameter constraints of BS-SVRG, which originally involves $\Theta(n)$th-order inequality. We provide the detailed parameter choice as follows:

\begin{propositions}[Parameter choice]
	\label{prop:R-params}
	In Algorithm \ref{alg:R-Acc-SVRG-G}, we set $\tau_x = \frac{\alpha + \delta_t}{\alpha + L+\delta_t}, \tau_z = \frac{\tau_x}{\delta_t} - \frac{\alpha(1 - \tau_x)}{\delta_t L}$ and $p = \frac{1}{n}$. We set $\alpha$ as the (unique) positive root of the cubic equation $\left(1 - \frac{p(\alpha + \delta_t)}{\alpha + L + \delta_t}\right)\left(1 + \frac{\delta_t}{\alpha}\right)^2 = 1$ and we specify $C_{\textup{IDC}} = L^2 + \frac{L\alpha^2p}{L + (1 -p)(\alpha + \delta_t)}, C_{\textup{IFC}} = 2L + \frac{2L\alpha^2p}{(L + (1 -p)(\alpha + \delta_t))\delta_t}$. Under these choices, we have $\frac{\alpha}{\delta_t} = O\big(n + \sqrt{n(L/\delta_t + 1)}\big), C_{\textup{IDC}} = O\big((L + \delta_t)^2\big)$ and $C_{\textup{IFC}} = O(L)$.
\end{propositions}

Under the choices of $\tau_x$ and $\tau_z$, the $\alpha$ above is the optimal choice in our analysis. Then, we can characterize the progress of the inner loop in the following proposition:

\begin{propositions}[The inner loop of Algorithm \ref{alg:R-Acc-SVRG-G}]
	\label{prop:R-Acc-SVRG-innr-loop}
	Using the parameters specified in Proposition \ref{prop:R-params}, after running the inner loop (Step \ref{alg-step:inner-loop}-\ref{alg-step:inner-loop-end}) of Algorithm~\ref{alg:R-Acc-SVRG-G} for $k$ iterations, we can conclude that 
	\begin{enumerate}[label=(\roman*),leftmargin=1cm, nosep]
	\item under IDC, i.e., $\norm{x_0 - \xs} \leq R_0 \,\text{ for some } \xs \in \Xs$,
	\[
	\E{\norm{\pf{\tilde{x}_k}}} \leq{} \left(\delta_t + \left(1 +\frac{\delta_t}{\alpha}\right)^{-k} \sqrt{C_{\textup{IDC}}}\right) R_0,
	\]
	\item under IFC, i.e., $f(x_0) - f(\xs) \leq \Delta_0$,
	\[
	\E{\norm{\pf{\tilde{x}_k}}} \leq{} \left(\sqrt{2\delta_t} + \left(1 +\frac{\delta_t}{\alpha}\right)^{-k} \sqrt{C_{\textup{IFC}}}\right)\sqrt{\Delta_0}.
	\]
\end{enumerate}
\end{propositions}
The above results motivate the construction of Step~\ref{alg-step:IDC-break} and \ref{alg-step:IFC-break}. For example, in the IDC setting, when the inner loop breaks at Step \ref{alg-step:IDC-break}, using \textit{(i)} above, we obtain $\E{\norm{\pf{\tilde{x}_k}}} \leq 2\delta_t R_0$. Then, by discussing the relative size of $\delta_t$ and a certain constant, we can estimate the complexity of Algorithm \ref{alg:R-Acc-SVRG-G}. The same methodology is used in the IFC setting.
\begin{theorems}[IDC case]\label{thm:R-Acc-SVRG-final-IDC}
	Denote $\delta^\star_\textup{IDC} = \frac{\epsilon q}{2R_0}$ for some $q\in (0,1)$ and let the outer iteration $t = \ell$ be the first time\footnote{We assume that $\epsilon$ is small such that $\max{\{\delta^\star_{\text{IDC}}, \delta^\star_{\text{IFC}}\}} \leq \delta_0 = L$ for simplicity. In this case, $\ell > 0$. } $\delta_\ell \leq \delta^\star_\textup{IDC}$. The following assertions hold.
	\begin{enumerate}[label=(\roman*),leftmargin=1cm, nosep]
		\item At outer iteration $\ell$, Algorithm \ref{alg:R-Acc-SVRG-G} terminates with probability at least $1 - q$.\footnote{If Algorithm \ref{alg:R-Acc-SVRG-G} does not terminate at outer iteration $\ell$, it terminates at the next outer iteration with probability at least $1 - q/\beta$. That is, it terminates with higher and higher probability. The same goes for the IFC case. }
		\item The total expected oracle complexity of the $\ell+1$ outer loops is 
		\[
		O\left(\left(n \log{\frac{L R_0}{\epsilon q}} + \sqrt{\frac{nLR_0}{\epsilon q}}\right)\log{\frac{LR_0}{\epsilon q} }\right).
		\]
	\end{enumerate}
\end{theorems}

\begin{theorems}[IFC case]\label{thm:R-Acc-SVRG-final-IFC}
	Denote $\delta^\star_\textup{IFC} = \frac{\epsilon^2 q^2}{8\Delta_0}$ for some $q\in (0,1)$ and let the outer iteration $t=\ell$ be the first time $\delta_\ell \leq \delta^\star_\textup{IFC}$. The following assertions hold.
	\begin{enumerate}[label=(\roman*),leftmargin=1cm, nosep]
		\item At outer iteration $\ell$, Algorithm \ref{alg:R-Acc-SVRG-G} terminates with probability at least $1 - q$.
		\item The total expected oracle complexity of the $\ell+1$ outer loops is 
		\[
		O\left(\left(n \log{\frac{\sqrt{L\Delta_0}}{\epsilon q}} + \frac{\sqrt{nL\Delta_0}}{\epsilon q}\right)\log{\frac{\sqrt{L\Delta_0}}{\epsilon q}}\right).
		\]
	\end{enumerate}
\end{theorems}

Compared with regularized Katyusha in Table \ref{gradient_table}, the adaptive regularization approach drops the need to estimate $R_0$ or $\Delta_0$ at the cost of a mere $\log{\frac{1}{\epsilon}}$ factor in the non-dominant term (if $\epsilon$ is small).

\section{Discussion}
\label{sec:discussion}
In this work, we proposed several simple and practical schemes that complement existing works (Table~\ref{gradient_table}). Admittedly, the new schemes are currently only limited to the unconstrained Euclidean setting, because our techniques heavily rely on the interpolation conditions \eqref{interpolation_c} and \eqref{interpolation_sc}. On the other hand, methods such as OGM \citep{OGM}, TM \citep{van2017fastest} and ITEM \citep{ITEM,d2021acceleration}, which also rely on these conditions, are still not known to have their proximal gradient variants. A concurrent work \citep{lee2021geometric} proposed proximal point variants of these algorithms. Extending their techniques to our schemes is left for future work. Another future work is to conduct extensive experiments to evaluate the proposed schemes. We list some other future directions as follows.

(1) It is not clear how to naturally connect the parameters of M-OGM-G (Algorithm \ref{alg:M-OGM-G}) to OGM-G (Algorithm \ref{alg:OGM-G}). The parameters of both algorithms seem to be quite restrictive and hardly generalizable due to the special construction at \eqref{T1}.

(2) Is this new ``momentum'' in OGM-G beneficial for training deep neural networks? Other classic momentum schemes such as NAG \citep{AGD1} or heavy-ball momentum method \citep{polyak1964some} are extremely effective for this task (see, e.g., \citep{pmlr-v28-sutskever13}), and they were also originally proposed for convex objectives.

(3) Can we directly accelerate SARAH (L2S)? It seems that existing acceleration techniques fail to accelerate SARAH (or result in poor dependence on $n$ as in \citep{acc-vr}). According to its position in Table \ref{gradient_table}, we suspect that there exists an accelerated variant of SARAH which reduces to OGM-G when $n = 1$.

\bibliography{MGS_finite_sums}
\bibliographystyle{abbrvnat}

\newpage
\appendix

\section{Numerical results of Acc-SVRG-G (Algorithm \ref{alg:Acc-SVRG-G})}
\label{app:exp}
\begin{figure}[H]
	\begin{center}
		\begin{subfigure}{0.4\linewidth}
			\includegraphics[width=\linewidth]{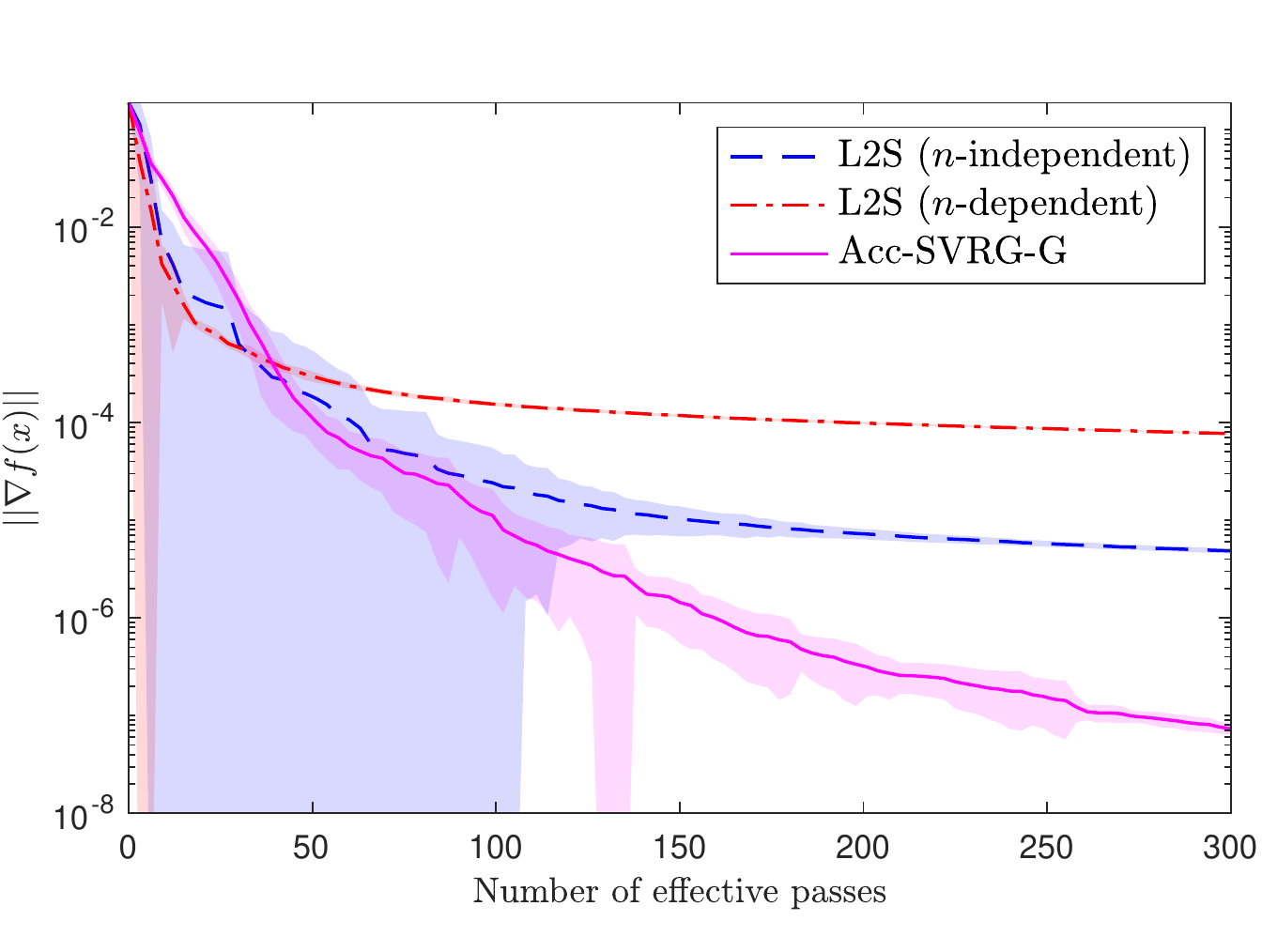}
			\caption{\textsf{a9a} dataset. Measuring the gradient norm.}
			\label{a9a_norm}
		\end{subfigure} \qquad\qquad
		\begin{subfigure}{0.4\linewidth}
			\includegraphics[width=\linewidth]{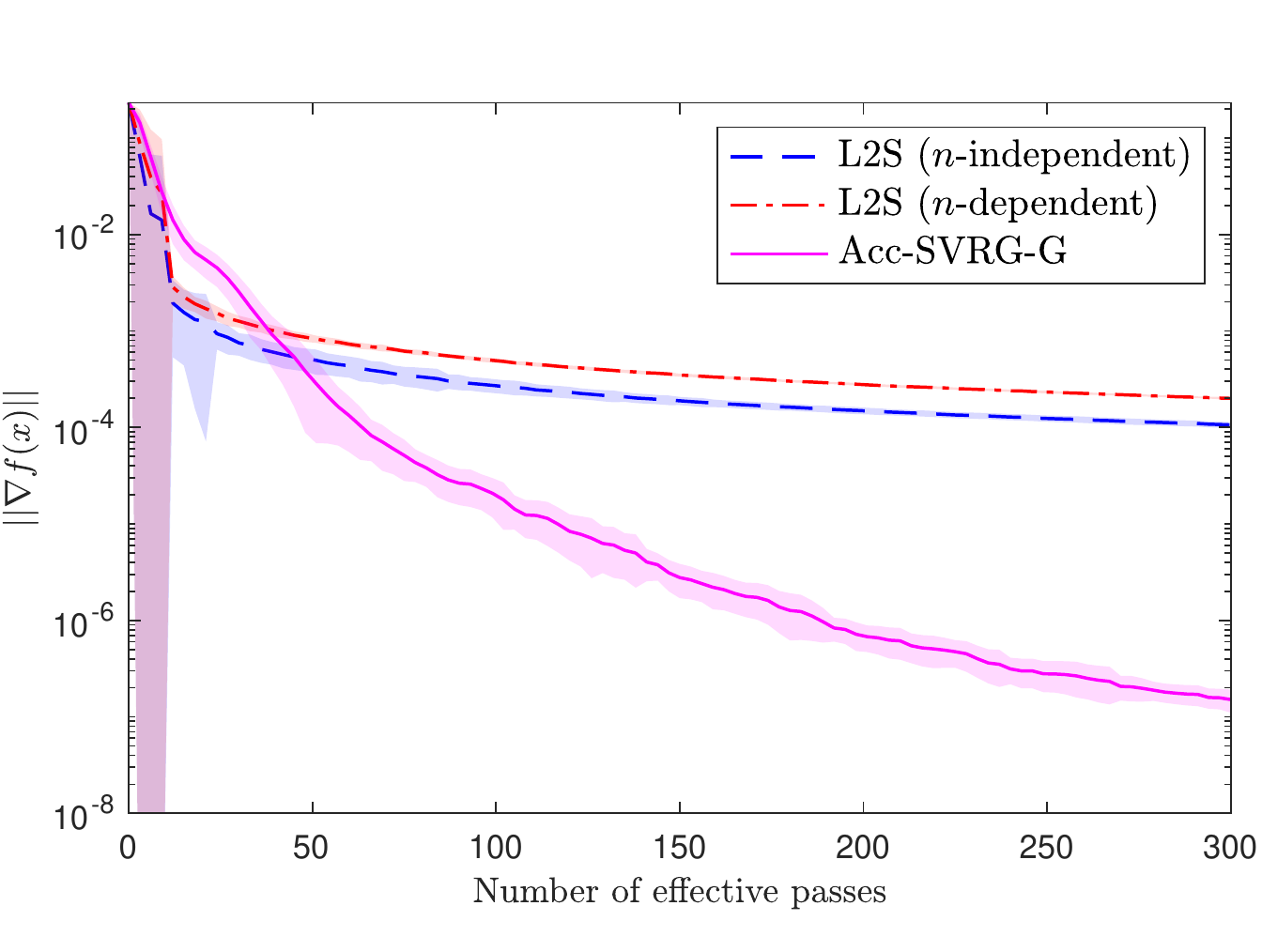}
			\caption{\textsf{w8a} dataset. Measuring the gradient norm.}
			\label{w8a_norm}
		\end{subfigure} \\
		\begin{subfigure}{0.4\linewidth}
			\includegraphics[width=\linewidth]{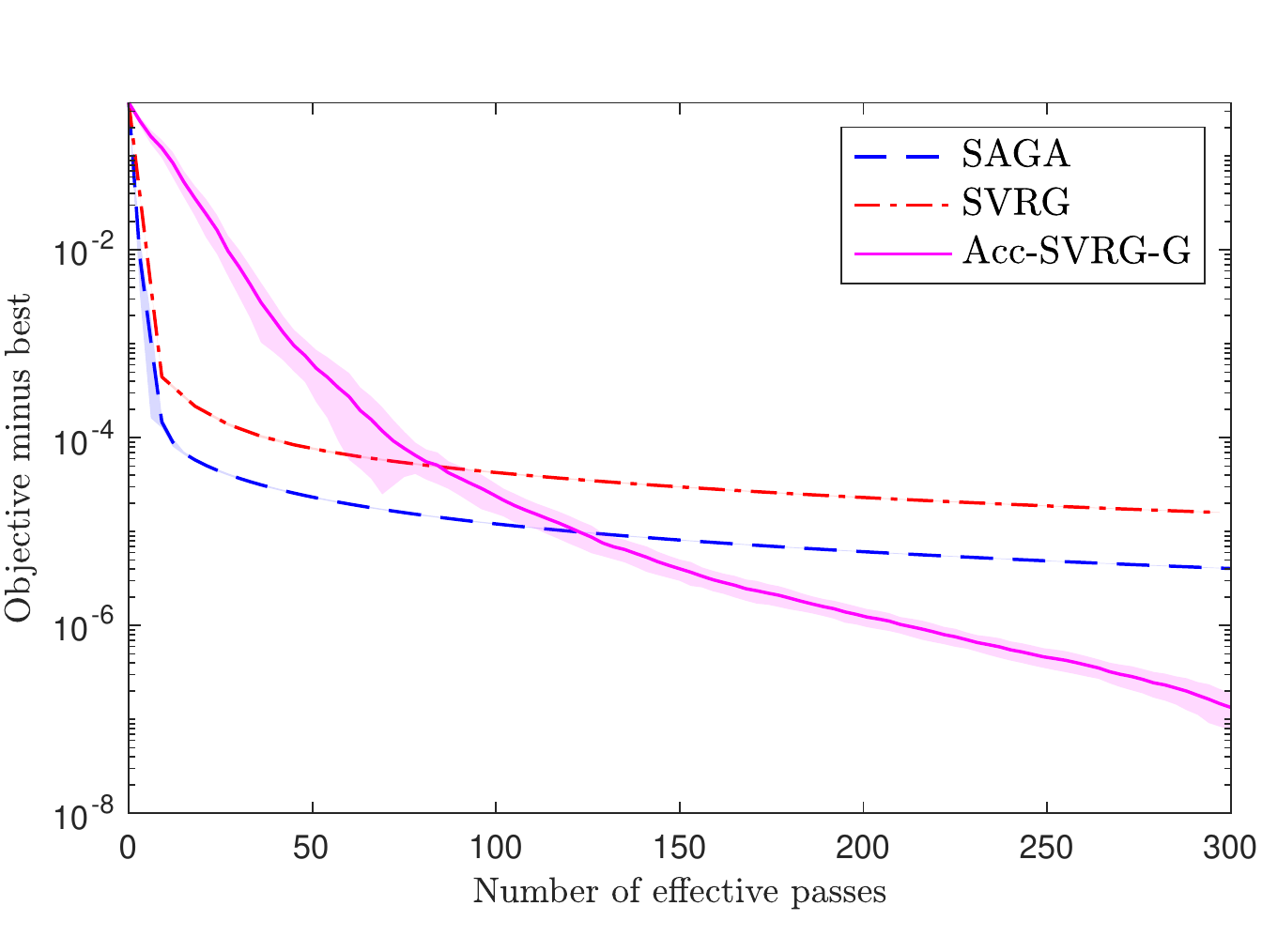}
			\caption{\textsf{a9a} dataset. Measuring the function value.}
			\label{a9a_function}
		\end{subfigure} \qquad\qquad
		\begin{subfigure}{0.4\linewidth}
			\includegraphics[width=\linewidth]{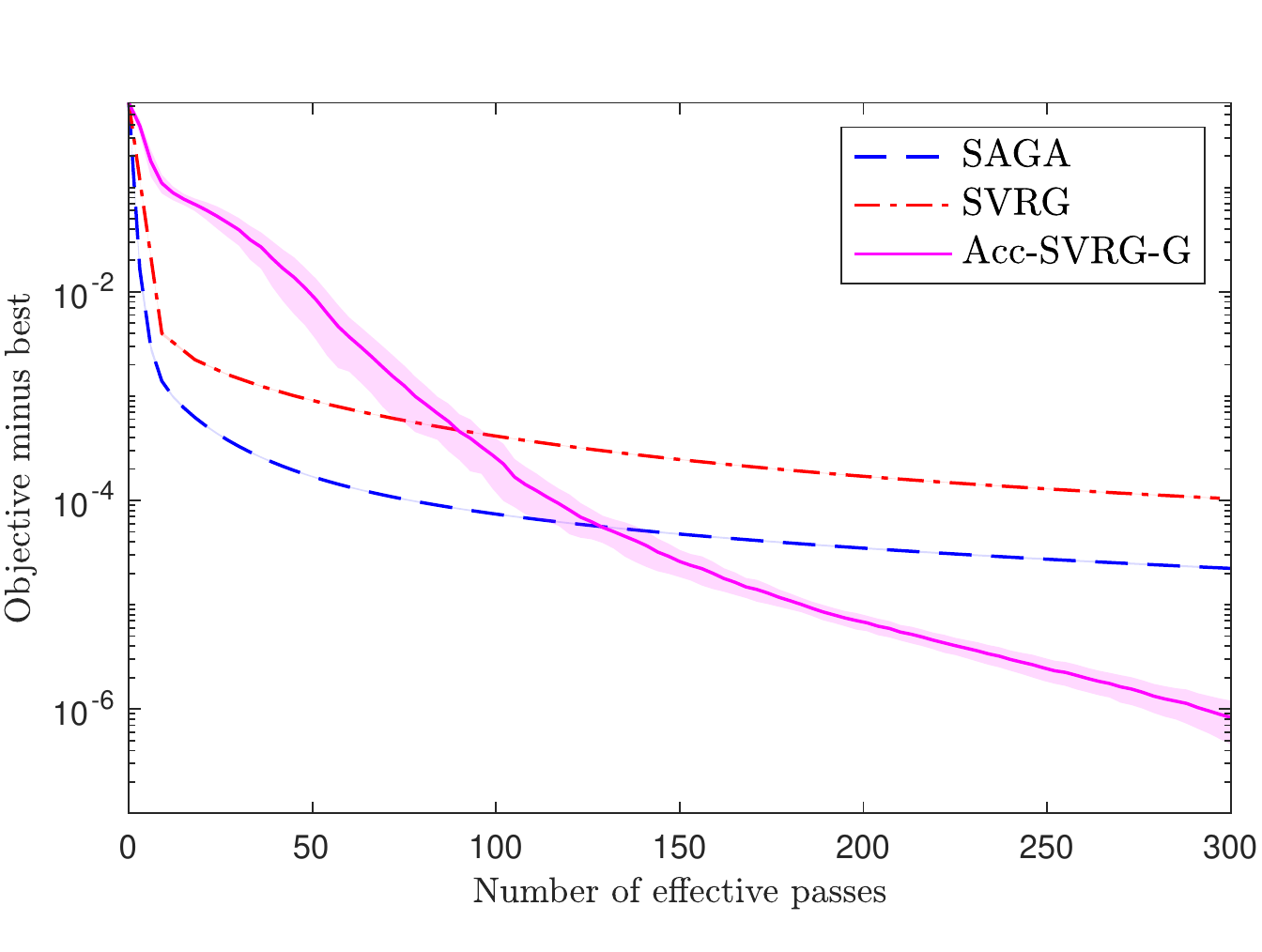}
			\caption{\textsf{w8a} dataset. Measuring the function value.}
			\label{w8a_function}
		\end{subfigure}
		\caption{Performance evaluations. Run $20$ seeds. Shaded bands indicate $\pm 1$ standard deviation.}
		\label{Eval_Acc_SVRG_G}
	\end{center}
\end{figure}

We did some experiments to justify the theoretical results (Theorem \ref{thm:Acc-SVRG-G-multiple-stage}) of Acc-SVRG-G. We compared it with non-accelerated methods including L2S \citep{li20a}, SVRG \citep{SVRG,Prox_SVRG} and SAGA \citep{SAGA} under their original optimality measures. Note that other stochastic approaches in Table \ref{gradient_table} require fixing the accuracy $\epsilon$ in advance, and thus it is not convenient to compare them in the form of Figure \ref{Eval_Acc_SVRG_G}. For measuring the gradient norm, we simply tracked the smallest norm of all the full gradient computed to reduce complexity. Since the figures are in logarithmic scale, the deviation bands are asymmetric, and will emphasize the passes that have large deviations.

\textbf{Setups.} We ran the experiments on a Macbook Pro with a quad-core Intel Core
i7-4870HQ with 2.50GHz cores, 16GB RAM, macOS Big Sur with Clang 12.0.5 and MATLAB R2020b. We were optimizing the binary logistic regression problem $
f(x) = \frac{1}{n} \sum_{i=1}^n {\log{\big(1 + \exp{(-b_i\innr{a_i,x})}\big)}}$ with dataset $a_i \in \R^d$, $b_i \in \lbrace -1, +1\rbrace$, $i\in[n]$. We used datasets from the LIBSVM website \citep{LIBSVM}, including \textsf{a9a} \citep{Dua:2019} (32,561 samples, 123 features) and \textsf{w8a} \citep{platt1998sequential} (49,749 samples, 300 features). We added one dimension as bias to all the datasets. We normalized the datasets and thus for this problem, $L = 0.25$. For Acc-SVRG-G, we chose the parameters according to Theorem \ref{thm:Acc-SVRG-G-multiple-stage}. For L2S, we set $m=n$ and for its $n$-independent step size, we chose $\eta = \frac{c}{L}$ and tuned $c$ using the same grid specified in \citep{li20a}; for the $n$-dependent step size, we set $\eta = \frac{1}{L\sqrt{n}}$ according to Corollary 3 in \citep{li20a}. For SAGA \citep{SAGA}, we chose $\eta = \frac{1}{3L}$ following its theory. For SVRG \citep{Prox_SVRG}, we set $\eta = \frac{1}{4L}$.

\newpage
\section{Proofs of \label{key}Section \ref{sec:OGM-G}}
\label{app:proof-ogm}
To simplify the proof, we denote $D_k \triangleq f(x_k) - f(\xs)$. And we use the following reformulation of interpolation condition \eqref{interpolation_c} (at $(x,y)$) to facilitate our proof.
\begin{equation}\label{interpolation_cr}
	\forall x, y\in \R^d,
	\frac{1}{2L} \left(\norm{\pf{x}}^2 + \norm{\pf{y}}^2\right) + \innr{\pf{y}, x - y- \frac{1}{L}\pf{x}} \leq f(x) - f(y).
\end{equation}

\subsection{Proof to Proposition \ref{prop:OGM-G}}

We define $\theta_{N+1}^2 = \theta_N^2 - \theta_N = 0$. At iteration $k$, we are going to combine the reformulated interpolation conditions \eqref{interpolation_cr} at $(x_k, x_{k+1})$ and $(x_N, x_k)$ with multipliers $\frac{1}{\theta_{k+1}^2}$ and $\frac{1}{\theta_k\theta_{k+1}^2}$, respectively.
\begin{align}
&\begin{aligned}
	&\frac{1}{2L\theta_{k+1}^2}\left(\norm{\pf{x_k}}^2 +  \norm{\pf{x_{k+1}}}^2\right) \!+\! \frac{1}{\theta_{k+1}^2}\innr{\pf{x_{k+1}}, x_k - x_{k+1} - \frac{1}{L} \pf{x_k}} \\\leq{}& \frac{1}{\theta_{k+1}^2}(D_k - D_{k+1}),
\end{aligned}\label{P1}\\
&\begin{aligned}
	&\frac{1}{2L\theta_k\theta_{k+1}^2}\left(\norm{\pf{x_N}}^2 + \norm{\pf{x_k}}^2\right) + \frac{1}{\theta_k\theta_{k+1}^2}\innr{\pf{x_k}, x_N - x_k - \frac{1}{L}\pf{x_N}} \\\leq{}& \frac{1}{\theta_k\theta_{k+1}^2}(D_N - D_k).
\end{aligned}\label{P2}
\end{align}

Using the construction: $x_k - x_{k+1} = \frac{1}{L}\pf{x_k} + (2\theta_{k+1}^3 - \theta_{k+1}^2) v_{k+1}$, we can write \eqref{P1} as
\begin{equation}\label{P3}
	\frac{1}{2L\theta_{k+1}^2}\left(\norm{\pf{x_k}}^2 +  \norm{\pf{x_{k+1}}}^2\right) + (2\theta_{k+1} - 1)\innr{\pf{x_{k+1}},  v_{k+1}} \leq{} \frac{1}{\theta_{k+1}^2}(D_k - D_{k+1}).
\end{equation}

Note that using $\theta_k^2 - \theta_k = \theta_{k+1}^2$, we have $2\theta_{k+1}^3 - \theta_{k+1}^2 = \theta_{k+1}^4 - \theta_{k+2}^4$. Then, 
\[
\begin{aligned}
	x_k - x_N = \sum_{i=k}^{N-1} {(x_i - x_{i+1})} &=  \frac{1}{L}\sum_{i=k}^{N-1}{\pf{x_i}} + \sum_{i=k}^{N-1}{(\theta_{i+1}^4 - \theta_{i+2}^4) v_{i+1}}\\
	&= \frac{1}{L}\sum_{i=k}^{N-1}{\pf{x_i}} + \theta_{k+1}^4 v_{k+1} + \sum_{i=k}^{N-2}{\theta_{i+2}^4 (v_{i+2} - v_{i+1})} \\
	&\meq{a} \frac{1}{L}\sum_{i=k}^{N-1}{\pf{x_i}} + \theta_{k+1}^4 v_{k+1} + \sum_{i=k}^{N-2}{ \frac{\theta_{i+2}^2}{L\theta_{i+1}}\pf{x_{i+1}}} \\
	&\meq{b} \theta_{k+1}^4 v_k + \sum_{i=k}^{N-1}{ \frac{\theta_i}{L}\pf{x_i}},
\end{aligned}
\]
where $\mar{a}$ and $\mar{b}$ use the construction: $v_{k+1} = v_k + \frac{1}{L\theta_k\theta_{k+1}^2}\pf{x_k}$.

Thus, \eqref{P2} can be written as
\[
\begin{aligned}
	\frac{1}{\theta_k\theta_{k+1}^2}(D_N - D_k)
	\geq{}& \frac{1}{2L\theta_k\theta_{k+1}^2}\norm{\pf{x_N}}^2 - \frac{\theta_k^2 + \theta_{k+1}^2}{2L\theta_k^2\theta_{k+1}^2}\norm{\pf{x_k}}^2 \\&- \frac{\theta_{k+1}^2}{\theta_k}\innr{\pf{x_k}, v_k} -  \sum_{i=k+1}^{N}{\frac{\theta_i}{L\theta_k\theta_{k+1}^2}\innr{\pf{x_k},\pf{x_i}}}.
\end{aligned}
\]
Summing this inequality and \eqref{P3}, and using the relation $\theta_k^2 - \theta_k = \theta_{k+1}^2$, we obtain

\begin{equation}\label{P5}
	\begin{aligned}
		&\left(\frac{1}{\theta_{k+1}^2} - \frac{1}{\theta_k^2}\right)\left(D_N - \frac{1}{2L}\norm{\pf{x_N}}^2\right) + \left(\frac{1}{\theta_k^2} D_k - \frac{1}{\theta_{k+1}^2} D_{k+1}\right)\\ 
		\geq{}& \left(\frac{1}{2L\theta_{k+1}^2}\norm{\pf{x_{k+1}}}^2  - \frac{1}{2L\theta_k^2}\norm{\pf{x_k}}^2\right)\\ &+\left(\frac{\theta_{k+2}^2}{\theta_{k+1}}\innr{\pf{x_{k+1}},  v_{k+1}} - \frac{\theta_{k+1}^2}{\theta_k}\innr{\pf{x_k}, v_k}\right)  \\
		&+ \underbrace{\theta_{k+1}\innr{\pf{x_{k+1}},  v_{k+1}} -  \sum_{i=k+1}^{N}{\frac{\theta_i}{L\theta_k\theta_{k+1}^2}\innr{\pf{x_k},\pf{x_i}}}}_{\mathcal{R}_1}.
	\end{aligned}
\end{equation}
\subsection{Proof to Theorem \ref{thm:OGM-G}}
\label{app:proof-OGM-G}
It is clear that except for $\mathcal{R}_1$, all terms in \eqref{P5} telescope. Since $v_{k+1} = \sum_{i=0}^{k} {\frac{1}{L\theta_i\theta_{i+1}^2}\pf{x_{i}}}$, by defining a matrix $P\in \R^{(N+1)\times (N+1)}$ with $P_{ki} = \frac{\theta_k}{L\theta_i\theta_{i+1}^2}\innr{\pf{x_k},  \pf{x_{i}}}$, we can write $\mathcal{R}_1$ as
$
\sum_{i=0}^{k} {P_{(k+1)i}} -  \sum_{i=k+1}^{N}{P_{ik}}.
$
Summing these terms from $k=0$ to $N-1$, we obtain
\[
	\sum_{k=0}^{N-1}{\sum_{i=0}^{k} {P_{(k+1)i}}} -  \sum_{k=0}^{N-1}{\sum_{i=k+1}^{N}{P_{ik}}}
	= \sum_{k=1}^{N}{\sum_{i=0}^{k-1} {P_{ki}}} -  \sum_{i=0}^{N-1}{\sum_{k=i+1}^{N}{P_{ki}}} = 0.
\]
Both of the summations are equal to the sum of the lower triangular entries of $P$. 

Then, telescoping \eqref{P5} from $k=0$ to $N-1$ (note that $v_0 = \mathbf{0}$), we obtain
\[
	\left(1 - \frac{1}{\theta_0^2}\right)\left(D_N - \frac{1}{2L}\norm{\pf{x_N}}^2\right) 
	\geq{} D_N - \frac{1}{\theta_0^2} D_0  + \frac{1}{2L}\norm{\pf{x_N}}^2-\frac{1}{2L\theta_0^2}\norm{\pf{x_0}}^2.
\]
Using $D_0 \geq \frac{1}{2L}\norm{\pf{x_0}}^2$ and $D_N \geq \frac{1}{2L}\norm{\pf{x_N}}^2$, we obtain
\[
\norm{\pf{x_N}}^2\leq \frac{2L D_0}{\theta_0^2}. 
\] 
Since $\theta_k = \frac{1 + \sqrt{1 + 4\theta_{k+1}^2}}{2} \geq \frac{1}{2} + \theta_{k+1} \Rightarrow \theta_k \geq \frac{N-k}{2} + 1 \Rightarrow \theta_0 \geq \frac{N+2}{2}$, we have
\[
\norm{\pf{x_N}}^2\leq \frac{8L \big(f(x_0) - f(\xs)\big)}{(N+2)^2}.
\]

\subsection{Proof to Theorem \ref{thm:M-OGM-G}}
Define for $k = 0, \ldots, N$,
\[
\begin{aligned}
	\tau_k \triangleq \frac{(N-k+2)(N-k+3)}{6},\ \ \delta_{k+1} \triangleq \frac{12}{(N-k+1)(N-k+2)(N-k+3)}= \frac{1}{\tau_{k+1}} - \frac{1}{\tau_k}.
\end{aligned}
\]
At iteration $k$, we are going to combine the reformulated interpolation conditions \eqref{interpolation_cr} at $(x_k, x_{k+1})$ and $(x_N, x_k)$ with multipliers $\frac{1}{\tau_{k+1}}$ and $\delta_{k+1}$, respectively.
\begin{align}
&\frac{1}{2L\tau_{k+1}}\left(\norm{\pf{x_k}}^2 +  \norm{\pf{x_{k+1}}}^2\right) \!+\! \frac{1}{\tau_{k+1}}\innr{\pf{x_{k+1}}, x_k \!-\! x_{k+1} \!-\! \frac{1}{L} \pf{x_k}} \leq{} \frac{1}{\tau_{k+1}}(D_k \!-\! D_{k+1}),\label{P11}\\
&\frac{\delta_{k+1}}{2L}\left(\norm{\pf{x_N}}^2 + \norm{\pf{x_k}}^2\right) + \delta_{k+1}\innr{\pf{x_k}, x_N - x_k - \frac{1}{L}\pf{x_N}} \leq{} \delta_{k+1}(D_N - D_k). \label{P12}
\end{align}

Note that from the construction of Algorithm \ref{alg:M-OGM-G},
\[
\begin{gathered}
	x_k - x_{k+1} -\frac{1}{L}\pf{x_k} = \frac{(N-k)(N-k+1)(N-k+2)}{6} v_{k+1},\\
	x_k - x_N = \sum_{i=k}^{N-1}{\frac{1}{L}\pf{x_i}} + \sum_{i=k}^{N-1}{\frac{(N-i)(N-i+1)(N-i+2)}{6} v_{i+1}}.
\end{gathered}
\]
Thus, \eqref{P11} can be written as
\begin{equation}\label{P111}
	\frac{1}{2L\tau_{k+1}}\left(\norm{\pf{x_k}}^2 +  \norm{\pf{x_{k+1}}}^2\right) + (N-k)\innr{\pf{x_{k+1}}, v_{k+1}}\leq \frac{1}{\tau_{k+1}}(D_k - D_{k+1}).
\end{equation}

Defining $\mathcal{Q}(j) \triangleq (j+3)(j+2)(j+1)j$, we have $\mathcal{Q}(j) - \mathcal{Q}(j-1) = 4j(j+1)(j+2)$. Then,
\[
\begin{aligned}
	x_k - x_N &= \sum_{i=k}^{N-1}{\frac{1}{L}\pf{x_i}} + \frac{1}{24}\sum_{i=k}^{N-1}{(\mathcal{Q}(N-i) - \mathcal{Q}(N-i-1)) v_{i+1}}\\
	&= \sum_{i=k}^{N-1}{\frac{1}{L}\pf{x_i}} + \frac{1}{24}\left(\mathcal{Q}(N-k)v_{k+1} + \sum_{i=k+1}^{N-1}{\mathcal{Q}(N-i)(v_{i+1} - v_i)}\right)\\
	&\meq{a}\frac{\mathcal{Q}(N-k)}{24} v_{k+1} + \frac{1}{L} \pf{x_k} + \sum_{i=k+1}^{N-1} {\frac{1}{L}\left(\frac{\mathcal{Q}(N-i)\delta_{i+1}}{24} + 1\right) \pf{x_i}}\\
	&\meq{b} \frac{\mathcal{Q}(N-k)}{24} v_k + \sum_{i=k}^{N-1} {\frac{N-i+2}{2L} \pf{x_i}},
\end{aligned}
\]
where $\mar{a}$ and $\mar{b}$ use the construction $v_{k+1} = v_k + \frac{\delta_{k+1}}{L}\pf{x_k}$.

Thus, \eqref{P12} can be written as
\[
\begin{aligned}
	\delta_{k+1}(D_N - D_k)
	\geq{}& \frac{\delta_{k+1}}{2L}\left(\norm{\pf{x_N}}^2 + \norm{\pf{x_k}}^2\right) -\frac{N-k}{2} \innr{\pf{x_k}, v_k}  \\ &- \frac{(N-k+2)\delta_{k+1}}{2L} \norm{\pf{x_k}}^2 - \sum_{i=k+1}^{N} {\frac{(N-i+2)\delta_{k+1}}{2L}\innr{\pf{x_k}, \pf{x_i}}}.
\end{aligned}
\]
Summing the above inequality and \eqref{P111}, we obtain
\begin{equation}\label{P19}
\begin{aligned}
	&\left(\frac{1}{\tau_{k+1}} - \frac{1}{\tau_k}\right) \left(D_N- \frac{1}{2L}\norm{\pf{x_N}}^2\right) + \left(\frac{1}{\tau_k} D_k - \frac{1}{\tau_{k+1}}D_{k+1}\right)\\	\geq{}& \left(\frac{1}{2L\tau_{k+1}} \norm{\pf{x_{k+1}}}^2 - \frac{1}{2L\tau_k}\norm{\pf{x_k}}^2\right) + \frac{\delta_{k+1}}{2L}\norm{\pf{x_k}}^2 \\
	&+ \left(\frac{N-k-1}{2} \innr{\pf{x_{k+1}},v_{k+1}} - \frac{N-k}{2}\innr{\pf{x_k},v_k}\right) \\ &+ \frac{N-k+1}{2} \innr{\pf{x_{k+1}},v_{k+1}} - \sum_{i=k+1}^N {\frac{(N-i+2)\delta_{k+1}}{2L}\innr{\pf{x_k}, \pf{x_i}}}.
\end{aligned}
\end{equation}
Since 
$
v_{k+1} = \sum_{i=0}^{k} {\frac{\delta_{i+1}}{L}\pf{x_i}}
$, the last two terms above have a similar structure as $\mathcal{R}_1$ at \eqref{P5}. Define a matrix $P\in \R^{(N+1)\times (N+1)}$ with $P_{ki} = \frac{(N-k+2)\delta_{i+1}}{2L}\innr{\pf{x_k},\pf{x_i}}$. The last two terms above can be written as $\sum_{i=0}^{k} {P_{(k+1)i}} -  \sum_{i=k+1}^{N}{P_{ik}}$. If we sum these terms from $k = 0,\ldots, N-1$, they sum up to $0$ (see Section \ref{app:proof-OGM-G}). Then, by telescoping \eqref{P19} from $k = 0,\ldots, N-1$, we obtain
\[
\begin{aligned}
	&\frac{1}{2L} \norm{\pf{x_N}}^2 - \frac{1}{2L\tau_0}\norm{\pf{x_0}}^2 + \frac{1 - \frac{1}{\tau_0}}{2L}\norm{\pf{x_N}}^2  + \sum_{k=0}^{N-1}{\frac{\delta_{k+1}}{2L}\norm{\pf{x_k}}^2} \\\leq{}& \left(1 - \frac{1}{\tau_0}\right) D_N + \frac{1}{\tau_0} D_0 - D_N.\\
\end{aligned}
\]

Finally, using $D_0 \geq \frac{1}{2L}\norm{\pf{x_0}}^2$ and $D_N \geq \frac{1}{2L}\norm{\pf{x_N}}^2$, we obtain
\begin{equation}\label{P20}
	\norm{\pf{x_N}}^2+ \sum_{k=0}^{N-1}{\frac{\delta_{k+1}}{2}\norm{\pf{x_k}}^2} \leq \frac{2L}{\tau_0} D_0 = \frac{12L \big(f(x_0) - f(\xs)\big)}{(N+2)(N+3)}.
\end{equation}

\subsection{Proof to Corollary \ref{coroll:NAG+M-OGM-G}}
We assume $N$ is divisible by $2$ for simplicity. After running $N/2$ iterations of NAG, we obtain an output $x_{N/2}$ satisfying  (cf. Theorem 2.2.2 in \citep{AGD3})
\[
	f(x_{N/2}) - f(\xs) = O\left(\frac{LR_0^2}{N^2}\right).
\] Then, let $x_{N/2}$ be the input of Algorithm \ref{alg:M-OGM-G}. Using \eqref{P20}, after running another $N/2$ iterations of Algorithm \ref{alg:M-OGM-G}, we obtain \[\norm{\pf{x_N}}^2 = O\left(\frac{L^2R_0^2}{N^4}\right).\]

\section{Proofs of Section \ref{sec:Acc-SVRG-G}}
\label{app:proof-acc-svrg-g}

\subsection{Proof to Proposition \ref{prop:Acc-SVRG-final}}

	Using the interpolation condition \eqref{interpolation_c} at $(\xs, y_k)$, we obtain
	\begin{align}
		f(y_k) - f(\xs) \leq{}& \innr{\pf{y_k}, y_k - \xs} - \frac{1}{2L}\norm{\pf{y_k}}^2 \nonumber \\
		\mleq{\star}{}& \frac{1 - \tau_k}{\tau_k}\innr{\pf{y_k}, \tilde{x}_k - y_k} - \frac{1 - \tau_k}{L\tau_k}\innr{\pf{y_k},\pf{\tilde{x}_k}}\label{P7} \\&+ \innr{\pf{y_k}, z_k - \xs} - \frac{1}{2L}\norm{\pf{y_k}}^2,\nonumber
	\end{align}
	where $\mar{\star}$ follows from the construction $y_{k} = \tau_k z_k + \left(1 - \tau_k\right) \left(\tilde{x}_k - \frac{1}{L}\pf{\tilde{x}_k}\right)$.
	
	From the optimality condition of Step \ref{alg-step:mirror_c}, we can conclude that
	\begin{align}
		&\mathcal{G}_k + \alpha_k (z_{k+1} - z_k) = \mathbf{0} \nonumber\\ \mra{a}{}&
		\innr{\mathcal{G}_k, z_k - \xs} = \frac{1}{2\alpha_k}\norm{\mathcal{G}_k}^2 + \frac{\alpha_k}{2} \left(\norm{z_k - \xs}^2 - \norm{z_{k+1} - \xs}^2\right) \nonumber\\
		\mra{b}{}&\innr{\pf{y_k}, z_k - \xs} = \frac{1}{2\alpha_k}\Eik{\norm{\mathcal{G}_k}^2} + \frac{\alpha_k}{2} \left(\norm{z_k - \xs}^2 - \Eik{\norm{z_{k+1} - \xs}^2}\right), \label{P8}
	\end{align}
	where $\mar{a}$ uses $\innr{u, v} = \frac{1}{2}(\norm{u}^2 + \norm{v}^2 - \norm{u - v}^2)$ and $\mar{b}$ follows from taking the expectation wrt sample $i_k$.
	
	Using the interpolation condition \eqref{interpolation_c} at $(\tilde{x}_k, y_k)$, we can bound $\Eik{\norm{\mathcal{G}_k}^2}$ as
	\begin{align}
		\Eik{\norm{\mathcal{G}_k}^2} ={}& \Eik{\norm{\pfik{y_k} - \pfik{\tilde{x}_k}}^2} + 2\innr{\pf{y_k}, \pf{\tilde{x}_k}} - \norm{\pf{\tilde{x}_k}}^2 \nonumber\\
		\leq{}& 2L\big(f(\tilde{x}_k) - f(y_k) - \innr{\pf{y_k}, \tilde{x}_k - y_k}\big) + 2\innr{\pf{y_k}, \pf{\tilde{x}_k}}  - \norm{\pf{\tilde{x}_k}}^2. \label{P9}
	\end{align}
	
	Combine \eqref{P7}, \eqref{P8} and \eqref{P9}. 
	\[
	\begin{aligned}
		f(y_k) - f(\xs) \leq{}& \frac{L}{\alpha_k}\big(f(\tilde{x}_k) - f(y_k) \big) +  \left(\frac{1 - \tau_k}{\tau_k} -  \frac{L}{\alpha_k}\right)\innr{\pf{y_k}, \tilde{x}_k - y_k} \\& + \left(\frac{1}{\alpha_k}- \frac{1 - \tau_k}{L\tau_k}\right)\innr{\pf{y_k}, \pf{\tilde{x}_k}} \\&+ \frac{\alpha_k}{2} \left(\norm{z_k - \xs}^2 - \Eik{\norm{z_{k+1} - \xs}^2}\right)\\
		& - \frac{1}{2L}\norm{\pf{y_k}}^2 - \frac{1}{2\alpha_k}\norm{\pf{\tilde{x}_k}}^2.
	\end{aligned}
	\]
	
	Substitute the choice $\alpha_k = \frac{L\tau_k}{1 -\tau_k}$.
	\begin{align}
		\frac{1-\tau_k}{\tau_k^2}\big(f(y_k) - f(\xs)\big) \leq{}& \frac{(1 - \tau_k)^2}{\tau_k^2}\big(f(\tilde{x}_k) - f(\xs) \big) + \frac{L}{2} \left(\norm{z_k - \xs}^2 - \Eik{\norm{z_{k+1} - \xs}^2}\right) \nonumber\\
		& - \frac{1 - \tau_k}{2L\tau_k}\norm{\pf{y_k}}^2 - \frac{(1 - \tau_k)^2}{2L\tau_k^2}\norm{\pf{\tilde{x}_k}}^2. \label{P18}
	\end{align}
	
	Note that by construction, $\E{f(\tilde{x}_{k+1})} = p_k\E{f(y_k)} + (1 - p_k)\E{f(\tilde{x}_k)}$, and thus
	\[
	\begin{aligned}
		\frac{1-\tau_k}{\tau_k^2p_k}\E{f(\tilde{x}_{k+1}) - f(\xs)} \leq{}& \frac{(1-\tau_kp_k)(1-\tau_k)}{\tau_k^2p_k}\E{f(\tilde{x}_k) - f(\xs)} \\
		&+ \frac{L}{2} \left(\E{\norm{z_k - \xs}^2} - \E{\norm{z_{k+1} - \xs}^2}\right)\\
		& - \frac{1 - \tau_k}{2L\tau_k}\E{\norm{\pf{y_k}}^2} - \frac{(1 - \tau_k)^2}{2L\tau_k^2}\E{\norm{\pf{\tilde{x}_k}}^2}.
	\end{aligned}
	\]

\subsection{Proof to Theorem \ref{thm:Acc-SVRG-G-single-stage}}

	It can be easily verified that under this choice ($p_k \equiv \frac{1}{n}, \tau_k = \frac{3}{k/n + 6}$), for any $k \geq 0, n \geq 1$,
	\[
	\frac{(1-\tau_{k+1}p_{k+1})(1-\tau_{k+1})}{\tau_{k+1}^2p_{k+1}} \leq \frac{1-\tau_k}{\tau_k^2p_k}.
	\]
	
	Then, using Proposition \ref{prop:Acc-SVRG-final}, after summing \eqref{P10} from $k = 0, \ldots, K-1$, we obtain
	\[
	\begin{aligned}
		&\frac{n(1-\tau_{K-1})}{\tau_{K-1}^2}\E{f(\tilde{x}_K) - f(\xs)} + \frac{L}{2} \E{\norm{z_K - \xs}^2} + \sum_{k=0}^{K-1}{\frac{(1 - \tau_k)^2}{2L\tau_k^2}\E{\norm{\pf{\tilde{x}_k}}^2}} \\ \leq{}& (2n-1)\big(f(x_0) - f(\xs)\big) + \frac{L}{2} \norm{x_0 - \xs}^2.
	\end{aligned}
	\]	
	
	Note that $\tau_k \leq \frac{1}{2}, \forall k$. We have the following two consequences of the above inequality.
	\[
	\begin{gathered}
		\E{f(\tilde{x}_K)} - f(\xs) \leq \tau_{K-1}^2\left(4\big(f(x_0) - f(\xs)\big) + \frac{L}{n} \norm{x_0 - \xs}^2\right), \\
		\begin{aligned}
			\E{\norm{\pf{x_{\text{out}}}}^2}={}&\frac{1}{\sum_{k=0}^{K-1}{\tau_k^{-2}}}\sum_{k=0}^{K-1}{\frac{1}{\tau_k^2}\E{\norm{\pf{\tilde{x}_k}}^2}} \\\leq{}& \frac{16nL\big(f(x_0) - f(\xs)\big) + 4L^2 \norm{x_0 - \xs}^2}{\sum_{k=0}^{K-1}{\tau_k^{-2}}}.
		\end{aligned} 
	\end{gathered}
	\]
	
	Substituting the parameter choice, we obtain
	\[
	\begin{gathered}
		\E{f(\tilde{x}_K)} - f(\xs) \leq \frac{36n^2\big(f(x_0) - f(\xs)\big) + 9nL \norm{x_0 - \xs}^2}{(K+6n-1)^2} = \epsilon_f,
		\\
		\E{\norm{\pf{x_{\text{out}}}}^2} \leq \frac{144nL\big(f(x_0) - f(\xs)\big) + 36L^2 \norm{x_0 - \xs}^2}{	\sum_{k=0}^{K-1}{\left(\frac{k}{n} + 6\right)^2}}.
	\end{gathered}
	\]
	
	Note that 
	\[
	\sum_{k=0}^{K-1}{\left(\frac{k}{n} + 6\right)^2} \geq \int_{0}^{K} {\left(\frac{x-1}{n} + 6\right)^2dx} = \frac{(K + 6n - 1)^3 - (6n-1)^3}{3n^2}.
	\]
	
	Thus, 
	\[
	\E{\norm{\pf{x_{\text{out}}}}}^2\leq\E{\norm{\pf{x_{\text{out}}}}^2} 
	\leq \frac{432n^3L\big(f(x_0) - f(\xs)\big) + 108n^2L^2 \norm{x_0 - \xs}^2}{(K + 6n - 1)^3-(6n-1)^3}=\epsilon_g^2.
	\]
	
	Since the expected iteration cost of Algorithm \ref{alg:Acc-SVRG-G} is $\E{\#\text{grad}_k} = p_k (n + 2) + (1 - p_k)2 = 3$, to guarantee $\E{\norm{\pf{x_{\text{out}}}}}\leq \epsilon_g$ and $\E{f(\tilde{x}_K)} - f(\xs) \leq \epsilon_f$, the total oracle complexities are $O\Big(\frac{n(L(f(x_0) - f(\xs)))^{1/3}}{\epsilon_g^{2/3}} + \frac{(nLR_0)^{2/3}}{\epsilon_g^{2/3}}\Big)$ and 
	$O\Big(n\sqrt{\frac{f(x_0) - f(\xs)}{\epsilon_f}} + \frac{\sqrt{nL}R_0}{\sqrt{\epsilon_f}}\Big)$, respectively.

\subsection{Proof to Theorem \ref{thm:Acc-SVRG-G-multiple-stage}}

	First, it can be verified that for any $k \geq 0, n \geq 1$, the following inequality holds.
	\[
	\frac{(1-\tau_{k+1}p_{k+1})(1-\tau_{k+1})}{\tau_{k+1}^2p_{k+1}} \leq \frac{1-\tau_k}{\tau_k^2p_k}.
	\]
	The verification can be done by considering the two cases: (i) $k+8 < 6n$, where $p_k = \frac{6}{k+8}, \tau_k = \frac{1}{2}$, (ii) $k+8 \geq 6n$, in which $p_k = \frac{1}{n}, \tau_k = \frac{3n}{k+8}$.
	
	Then, using Proposition \ref{prop:Acc-SVRG-final}, after summing \eqref{P10} from $k = 0, \ldots, K-1$, we obtain
	\[
	\begin{aligned}
		&\frac{1-\tau_{K-1}}{\tau_{K-1}^2p_{K-1}}\E{f(\tilde{x}_K) - f(\xs)} + \frac{L}{2} \E{\norm{z_K - \xs}^2} + \sum_{k=0}^{K-1}{\frac{(1 - \tau_k)^2}{2L\tau_k^2}\E{\norm{\pf{\tilde{x}_k}}^2}} \\
		\leq{}& \frac{5}{3}\big(f(x_0) - f(\xs)\big) + \frac{L}{2} \norm{x_0 - \xs}^2 \leq{} \frac{4}{3}LR_0^2.
	\end{aligned}
	\]
	
	Note that $\tau_k \leq \frac{1}{2}, \forall k$. We can conclude the following two consequences.
	\begin{gather}
		\E{f(\tilde{x}_K)} - f(\xs) \leq \frac{8}{3}\tau_{K-1}^2p_{K-1}LR_0^2, \label{P14}\\
		\E{\norm{\pf{x_{\text{out}}}}^2} = \frac{1}{\sum_{k=0}^{K-1}{\tau_k^{-2}}}\sum_{k=0}^{K-1}{\frac{1}{\tau_k^2}\E{\norm{\pf{\tilde{x}_k}}^2}}\leq \frac{32L^2R_0^2}{3\sum_{k=0}^{K-1}{\tau_k^{-2}}}.\label{P15}
	\end{gather}
	
	Now we consider two stages.  
	
	\textbf{Stage I (low accuracy stage):} $K + 8 \leq 6n$. In this stage, let the accuracies be $\epsilon_g^2 = \frac{8L^2R_0^2}{3K} \geq \frac{8L^2R_0^2}{3(6n-8)}$ and $\epsilon_f = \frac{4LR_0^2}{K+7} \geq \frac{4LR_0^2}{6n-1}$.  By substituting the parameter choice, \eqref{P14} and \eqref{P15} can be written as
	\[
	\begin{gathered}
		\E{f(\tilde{x}_K)} - f(\xs) \leq \frac{4LR_0^2}{K+7} = \epsilon_f, \\
		\E{\norm{\pf{x_{\text{out}}}}}^2\leq\E{\norm{\pf{x_{\text{out}}}}^2} \leq \frac{8L^2R_0^2}{3K} = \epsilon_g^2.
	\end{gathered}
	\]
	
	Note that the expected iteration cost of Algorithm \ref{alg:Acc-SVRG-G} is $\E{\#\text{grad}_k} = p_k (n + 2) + (1 - p_k)2 = np_k + 2$, and thus the total complexity in this stage is
	\[
	\sum_{k=0}^{K-1}{\E{\#\text{grad}_k}} = n \sum_{k=0}^{K-1}{\frac{6}{k+8}} + 2K \leq 6 n\log{(K+7)} + 12n = O(n\log{K}).
	\]
	
	Thus, the expected oracle complexities in this stage are $O(n\log{\frac{LR_0}{\epsilon_g}})$ and $O(n\log{\frac{LR_0^2}{\epsilon_f}})$, respectively.
	
	\textbf{Stage II (high accuracy stage):} $K+8 > 6n$. In this stage, Algorithm \ref{alg:Acc-SVRG-G} proceeds to find highly accurate solutions (i.e., $\epsilon_g^2 < \frac{8L^2R_0^2}{3(6n-8)}$ and $\epsilon_f < \frac{4LR_0^2}{6n-1}$). Substituting the parameter choice, we can write \eqref{P14} and \eqref{P15} as
	\begin{gather}
		\E{f(\tilde{x}_K)} - f(\xs) \leq \frac{24nLR_0^2}{(K+7)^2} = \epsilon_f, \label{P16}\\
		\E{\norm{\pf{x_{\text{out}}}}^2} \leq \frac{32L^2R_0^2}{3\left(24n - 28 + \sum_{k=6n-7}^{K-1}{\tau_k^{-2}}\right)} \mleq{\star} \frac{288n^2L^2R_0^2}{(K+7)^3 + 432n^3 - 756n^2} = \epsilon_g^2, \label{P17}
	\end{gather}
	where $\mar{\star}$ follows from 
	\[
	\sum_{k=6n-7}^{K-1}{\tau_k^{-2}} = \frac{1}{9n^2}\sum_{k=6n-7}^{K-1}{(k+8)^2} \geq \frac{1}{9n^2} \int_{6n-7}^K{(x+7)^2dx} = \frac{(K+7)^3}{27n^2} - 8n.
	\]
	
	Then, we count the expected complexity in this stage.
	\[
	\sum_{k=0}^{K-1}{\E{\#\text{grad}_k}} = n\left(\sum_{k=0}^{6n-8}{\frac{6}{k+8}} + \sum_{k=6n-7}^{K-1}{\frac{1}{n}}\right) + 2K \leq 6n\log{(6n)} + 3K -6n+7.
	\]
	
	Finally, combining with \eqref{P16} and \eqref{P17}, we can conclude that the total expected oracle complexities in this stage are $O\Big(n\log{n} + \frac{(nLR_0)^{2/3}}{\epsilon_g^{2/3}}\Big)$ and $O\Big(n\log{n} + \frac{\sqrt{nL}R_0}{\sqrt{\epsilon_f}} \Big)$, respectively.

\subsection{Proof to Theorem \ref{thm:Acc-SVRG-G-first-epoch}}

	We start at inequality \eqref{P18} in the proof of Proposition \ref{prop:Acc-SVRG-final}, which is the consequence of one iteration $k$ in Algorithm \ref{alg:Acc-SVRG-G}. Due to the constant choice of $\tau_k\equiv \tau$, we have
	\[
	\begin{aligned}
		f(y_k) - f(\xs) \leq{}& (1 - \tau)\big(f(\tilde{x}_k) - f(\xs) \big) + \frac{L\tau^2}{2(1-\tau)} \left(\norm{z_k - \xs}^2 - \Eik{\norm{z_{k+1} - \xs}^2}\right) \\
		& - \frac{\tau}{2L}\norm{\pf{y_k}}^2 - \frac{1 - \tau}{2L}\norm{\pf{\tilde{x}_k}}^2.
	\end{aligned}
	\]
	
	Since we fix $p_k \equiv p$ as a constant and terminate Algorithm \ref{alg:Acc-SVRG-G} at the first time $\tilde{x}_{k+1} = y_k$ (denoted as the iteration $N$), it is clear that the random variable $N$ follows the geometric distribution with parameter $p$, that is, for $ k=0,1,2,\ldots, \prob{N = k} = (1-p)^kp$. Moreover, since we have $\tilde{x}_N = \tilde{x}_{N-1} =\cdots=\tilde{x}_0 = x_0$, using the above inequality at iteration $N$, we obtain
	\[
	\begin{aligned}
		\E{f(\tilde{x}_{N+1})} -\! f(\xs) \!\leq{}& (1 - \tau)\big(f(x_0) -\! f(\xs) \big) + \frac{L\tau^2}{2(1-\tau)} \left(\E{\norm{z_N - \xs}^2 \!- \norm{z_{N+1} - \xs}^2}\right)\\& - \frac{\tau}{2L}\E{\norm{\pf{\tilde{x}_{N+1}}}^2} - \frac{1 - \tau}{2L}\norm{\pf{x_0}}^2\\
		\meq{\star}{}& (1 - \tau)\big(f(x_0) -\! f(\xs) \big) + \frac{L\tau^2p}{2(1-\tau)} \left(\norm{x_0 - \xs}^2 - \E{\norm{z_{N+1} - \xs}^2}\right)\\& - \frac{\tau}{2L}\E{\norm{\pf{\tilde{x}_{N+1}}}^2} - \frac{1 - \tau}{2L}\norm{\pf{x_0}}^2,
	\end{aligned}
	\]
	where $\mar{\star}$ follows from 
	\[
	\begin{aligned}
		\E{\norm{z_{N+1} - \xs}^2} ={}& \frac{1}{1-p}\left(\sum_{k=0}^{\infty} {(1-p)^kp\E{\norm{z_k - \xs}^2}} - p\norm{z_0 - \xs}^2\right) \\
		={}& \frac{1}{1 - p}\left(\E{\norm{z_N - \xs}^2} - p\norm{z_0 - \xs}^2\right).
	\end{aligned}
	\] 
	
	Thus, we can conclude that
	\[
	\begin{aligned}
		\E{f(\tilde{x}_{N+1})} - f(\xs) + \frac{\tau}{2L}\E{\norm{\pf{\tilde{x}_{N+1}}}^2}
		\leq{} \frac{L}{2}\left(1 - \tau + \frac{\tau^2p}{1-\tau}\right)R_0^2.
	\end{aligned}
	\]
	
	Note that $\E{N} = \frac{1-p}{p}$ and the total expected oracle complexity is $n + 2(\E{N}+1) = n+\frac{2}{p}$. We choose $p = \frac{1}{n}$, which leads to an $O(n)$ expected complexity. And we choose $\tau$ by minimizing the ratio $\left(1 - \tau + \frac{\tau^2p}{1-\tau}\right)$ wrt $\tau$. This gives $\tau = 1 - \frac{1}{\sqrt{n+1}} \geq \frac{1}{4}$ and
	\[
	\begin{aligned}
		\E{f(\tilde{x}_{N+1})} - f(\xs) + \frac{1}{8L}\E{\norm{\pf{\tilde{x}_{N+1}}}^2}
		\leq{} \frac{LR_0^2}{\sqrt{n+1}+1}.
	\end{aligned}
	\]

\section{Proofs of Section \ref{sec:R-Acc-SVRG-G}}
\label{app:proof-r-acc-svrg-g}
We analyze Algorithm \ref{alg:R-Acc-SVRG-G} following the ``shifting'' methodology in \citep{zhou2020boosting}, which explores the tight interpolation condition \eqref{interpolation_sc} and leads to a simple and clean proof.

Note that after the regularization at Step \ref{alg-step:regularization}, each $f^{\delta_t}_i$ is $(L+\delta_t)$-smooth and $\delta_t$-strongly convex. We denote $\xsd$ as the unique minimizer of $\min_x{f^{\delta_t}(x)}$. Following \citep{zhou2020boosting}, we define a ``shifted'' version of this problem: $\min_{x} {h^{\delta_t}(x) = \frac{1}{n} \sum_{i=1}^n {h^{\delta_t}_i(x)}}$, where
\[
h^{\delta_t}_i(x) = f^{\delta_t}_i(x) - f^{\delta_t}_i(\xs_{\delta_t}) - \innr{\pfdii{\xs_{\delta_t}}, x - \xs_{\delta_t}} - \frac{\delta_t}{2} \norm{x - \xs_{\delta_t}}^2, \forall i.
\]
It can be easily verified that each $h^{\delta_t}_i$ is $L$-smooth and convex. Note that $h^{\delta_t}_i(\xsd) = h^{\delta_t}(\xsd) = 0$ and $ \phdii{\xsd} = \phd{\xsd} = \mathbf{0}$, which means that $h^{\delta_t}$ and $f^{\delta_t}$ share the same minimizer $\xsd$.

Then, conceptually, we attempts to solve the ``shifted'' problem using an ``shifted'' SVRG gradient estimator: $\mathcal{H}^{\delta_t}_k \triangleq \phdik{y_k} - \phdik{\tilde{x}_k} + \phd{\tilde{x}_k}$. Clearly, the gradient of $h^{\delta_t}$ is not accessible due to the unknown $\xsd$. \citet{zhou2020boosting} proposed a technical lemma (Lemma \ref{lem:shifting} below) to bypass this issue. Since the relation $\mathcal{H}^{\delta_t}_k = \mathcal{G}^{\delta_t}_k - \delta_t (y_k - \xsd)$ holds, we can use Lemma \ref{lem:shifting} as an instantiation of the ``shifted'' gradient oracle, see \citep{zhou2020boosting} for details.

\subsection{Technical Lemmas}

\begin{lemmas}[Lemma 1 in \citep{zhou2020boosting}, the ``shifting'' technique]
	\label{lem:shifting}
	Given a gradient estimator $\mathcal{G}_y$ and vectors $z^+, z^-, y, \xs\in \R^d$, fix the updating rule $z^+ = \arg\min_{x} \big\{\innr{\mathcal{G}_y, x} + \frac{\alpha}{2} \norm{x - z^-}^2 + \frac{\delta}{2}	\norm{x - y}^2 \big\}$. Suppose that we have a shifted gradient estimator $\mathcal{H}_y$ satisfying the relation $\mathcal{H}_y = \mathcal{G}_y - \delta (y - \xs)$, it holds that
	\[
	\innr{\mathcal{H}_y, z^- - \xs} = \frac{\alpha}{2} \left(\norm{z^- - \xs}^2 - \left(1 + \frac{\delta}{\alpha}\right)^2\norm{z^+ - \xs}^2\right) + \frac{1}{2\alpha} \norm{\mathcal{H}_y}^2.
	\]
\end{lemmas}

\begin{lemmas}[The regularization technique \citep{nesterov2012make}]
	\label{lem:regular}
	For an $L$-smooth and convex function $f$ and $\delta > 0$, defining $f^\delta(x) = f(x) + \frac{\delta}{2} \norm{x - x_0}^2, \forall x$ and denoting $\xs_\delta$ as the unique minimizer of $f^\delta$, we have
	\begin{enumerate}[label=(\roman*),leftmargin=1cm, nosep]
		\item $f^\delta$ is $(L+\delta)$-smooth and $\delta$-strongly convex. 
		\item $f^\delta(x_0) - f^\delta(\xs_\delta) \leq f(x_0) - f(\xs)$.
		\item $\norm{x_0 - \xs_\delta}^2 \leq \norm{x_0 - \xs}^2, \forall \xs\in \Xs.$
		\item $\norm{x_0 - \xs_\delta}^2 \leq \frac{2}{\delta} \big(f(x_0) - f(\xs)\big)$.
	\end{enumerate}
\end{lemmas}
\begin{proof}
	\textit{(i)} can be easily checked by the definition of $L$-smoothness and strong convexity. \textit{(ii)} follows from $f^\delta(x_0) = f(x_0)$ and $f^\delta(\xs_\delta) \geq f(\xs_\delta) \geq f(\xs)$. For \textit{(iii)}, using the strong convexity of $f^\delta$ at $(\xs, \xs_\delta), \forall \xs \in \Xs$, we obtain
	\[
	\begin{aligned}
		&f^\delta(\xs) - f^\delta(\xs_\delta) \geq \frac{\delta}{2} \norm{\xs - \xs_\delta}^2\\ \Rightarrow{}&  f(\xs) + \frac{\delta}{2} \norm{\xs - x_0}^2 - f(\xs_\delta) - \frac{\delta}{2} \norm{\xs_\delta - x_0}^2 \geq \frac{\delta}{2} \norm{\xs - \xs_\delta}^2\\
		\Rightarrow{}& \frac{\delta}{2} \norm{x_0 - \xs}^2 - \big(f(\xs_\delta) - f(\xs)\big) \geq \frac{\delta}{2} \norm{x_0 - \xs_\delta}^2 + \frac{\delta}{2} \norm{\xs - \xs_\delta}^2.
	\end{aligned}
	\]
	Then \textit{(iii)} follows from the non-negativeness of $f(\xs_\delta) - f(\xs)$ and $\norm{\xs - \xs_\delta}^2$. For \textit{(iv)}, using the strong convexity of $f^\delta$ at $(x_0, \xs_\delta)$ and \textit{(ii)}, we have $\norm{x_0 - \xs_\delta}^2 \leq \frac{2}{\delta} \big(f^\delta(x_0) - f^\delta(\xs_\delta)\big) \leq \frac{2}{\delta} \big(f(x_0) - f(\xs)\big)$.
\end{proof}

\subsection{Proof to Proposition \ref{prop:R-params}}
	Denoting $\kappa_t = \frac{L+\delta_t}{\delta_t}$, we can write the equation $\left(1 - \frac{p(\alpha + \delta_t)}{\alpha + L + \delta_t}\right)\left(1 + \frac{\delta_t}{\alpha}\right)^2 = 1$ as
	\[
	s\left(\frac{\alpha}{\delta_t}\right) \triangleq\left(\frac{\alpha}{\delta_t}\right)^3 - (2n-3)\left(\frac{\alpha}{\delta_t}\right)^2 - (2n\kappa_t + n - 3)\left(\frac{\alpha}{\delta_t}\right) - n\kappa_t + 1 = 0.
	\]
	It can be verified that $s(2n+2\sqrt{n\kappa_t}) > 0$ for any $n\geq 1, \kappa_t > 1$. Since $s(0) < 0$ and $s(\frac{\alpha}{\delta_t}) \rightarrow \infty$ as $\frac{\alpha}{\delta_t} \rightarrow \infty$, the unique positive root satisfies $\frac{\alpha}{\delta_t} \leq 2n + 2\sqrt{n\kappa_t} = O(n+\sqrt{n\kappa_t})$. 
	
	To bound $C_\textup{IDC}$ and $C_\textup{IFC}$, it suffices to note that
	\[
	\frac{\frac{\alpha^2}{\delta_t^2}p}{\frac{L}{\delta_t}+(1 - p)(\frac{\alpha}{\delta_t} + 1)} \meq{a} \frac{(\frac{\alpha}{\delta_t}+1)^2}{n(\frac{\alpha}{\delta_t}+\kappa_t)} \mleq{b} \frac{(2n + 2\sqrt{n\kappa_t}+1)^2}{n(2n + 2\sqrt{n\kappa_t}+\kappa_t)} \leq 6,
	\]
	where $\mar{a}$ uses the cubic equation and $\mar{b}$ holds because $\frac{x+1}{x+\kappa_t}$ increases monotonically as $x$ increases. Then, 
	\[
	\begin{aligned}
	C_\textup{IDC} &\leq L^2 + 6L\delta_t = O\big((L + \delta_t)^2\big), \\ 
	C_\textup{IFC} &\leq 14L = O(L).
	\end{aligned}
	\] 
	
\subsection{Proof to Proposition \ref{prop:R-Acc-SVRG-innr-loop}}
	Using the interpolation condition \eqref{interpolation_sc} of $h^{\delta_t}$ at $(\xsd, y_k)$, we obtain
	\[
	\begin{aligned}
		h^{\delta_t}(y_k) \leq{}& \innr{\phd{y_k}, y_k - \xsd} - \frac{1}{2L} \norm{\phd{y_k}}^2\\
		\mleq{a}{}& \frac{1 - \tau_x}{\tau_x}\innr{\phd{y_k}, \tilde{x}_k - y_k} + \frac{\tau_z}{\tau_x} \innr{\phd{y_k}, \delta_t(\tilde{x}_k - z_k) - \pfd{\tilde{x}_k}}\\
		&+ \innr{\phd{y_k}, z_k - \xsd} - \frac{1}{2L} \norm{\phd{y_k}}^2\\
		\meq{b}{}&\frac{1 - \tau_x}{\tau_x}\innr{\phd{y_k}, \tilde{x}_k - y_k} - \frac{\tau_z}{\tau_x} \innr{\phd{y_k}, \phd{\tilde{x}_k}}\\
		&+ \left(1 - \frac{\delta_t\tau_z}{\tau_x}\right)\innr{\phd{y_k}, z_k - \xsd} - \frac{1}{2L} \norm{\phd{y_k}}^2,
	\end{aligned}
	\]
	where $\mar{a}$ follows from the construction $y_{k} = \tau_x z_k + \left(1 - \tau_x\right) \tilde{x}_k +  \tau_{z}\left(\delta_t(\tilde{x}_k - z_k) - \pfd{\tilde{x}_k}\right)$ and $\mar{b}$ uses that $\delta_t(\tilde{x}_k - z_k) - \pfd{\tilde{x}_k} = \delta_t(\xsd - z_k) - \phd{\tilde{x}_k}$. 
	
	Using Lemma \ref{lem:shifting} with $\mathcal{H}_y = \mathcal{H}^{\delta_t}_k, \mathcal{G}_y = \mathcal{G}^{\delta_t}_k, z^+ = z_{k+1}, \xs = \xsd$ and taking the expectation (note that $\Eik{\mathcal{H}^{\delta_t}_k} = \phd{y_k}$), we can conclude that
	\[
	\begin{aligned}
		h^{\delta_t}(y_k) \leq{}& \frac{1 - \tau_x}{\tau_x}\innr{\phd{y_k}, \tilde{x}_k - y_k} - \frac{\tau_{z}}{\tau_x} \innr{\phd{y_k}, \phd{\tilde{x}_k}} - \frac{1}{2L} \norm{\phd{y_k}}^2 \\
		&+ \left(1 - \frac{\delta_t\tau_z}{\tau_x}\right)\frac{\alpha}{2} \left(\norm{z_k - \xsd}^2 - \left(1 + \frac{\delta_t}{\alpha}\right)^2\Eik{\norm{z_{k+1} - \xsd}^2}\right) \\&+ \left(1 - \frac{\delta_t\tau_z}{\tau_x}\right)\frac{1}{2\alpha} \Eik{\norm{\mathcal{H}^{\delta_t}_k}^2}.
	\end{aligned}
	\]
	
	To bound the shifted moment, we use the interpolation condition \eqref{interpolation_sc} of $h^{\delta_t}_{i_k}$ at $(\tilde{x}_k, y_k)$, that is
	\[
	\begin{aligned}
		\Eik{\norm{\mathcal{H}^{\delta_t}_k}^2} ={}& \Eik{\norm{\phdik{y_k} - \phdik{\tilde{x}_k}}^2} + 2\innr{\phd{y_k}, \phd{\tilde{x}_k}} - \norm{\phd{\tilde{x}_k}}^2\\
		\leq{}& 2L \big(h^{\delta_t}(\tilde{x}_k) - h^{\delta_t}(y_k) - \innr{\phd{y_k}, \tilde{x}_k - y_k}\big) \\ 
		&+ 2\innr{\phd{y_k}, \phd{\tilde{x}_k}} - \norm{\phd{\tilde{x}_k}}^2.
	\end{aligned}
	\]

	Re-arrange the terms.
	\[
	\begin{aligned}
		h^{\delta_t}(y_k) \leq{}& \left(1 - \frac{\delta_t\tau_z}{\tau_x}\right) \frac{L}{\alpha} \big(h^{\delta_t}(\tilde{x}_k) - h^{\delta_t}(y_k)\big) \\&+\left(\frac{1 - \tau_x}{\tau_x} - \left(1 - \frac{\delta_t\tau_z}{\tau_x}\right) \frac{L}{\alpha}\right)\innr{\phd{y_k}, \tilde{x}_k - y_k} \\
		&+ \left(1 - \frac{\delta_t\tau_z}{\tau_x}\right)\frac{\alpha}{2} \left(\norm{z_k - \xsd}^2 - \left(1 + \frac{\delta_t}{\alpha}\right)^2\Eik{\norm{z_{k+1} - \xsd}^2}\right) \\
		&+ \left(\frac{1}{\alpha} - \frac{\delta_t\tau_z}{\alpha\tau_x} - \frac{\tau_{z}}{\tau_x}\right) \innr{\phd{y_k}, \phd{\tilde{x}_k}}
		- \frac{1}{2L}\norm{\phd{y_k}}^2 \\
		&- \left(\frac{1}{2\alpha} - \frac{\delta_t\tau_z}{2\alpha\tau_x}\right)\norm{\phd{\tilde{x}_k}}^2.
	\end{aligned}
	\]
	
	The choice of $\tau_z$ in Proposition \ref{prop:R-params} ensures that $\frac{1 - \tau_x}{\tau_x} = \left(1 - \frac{\delta_t\tau_z}{\tau_x}\right) \frac{L}{\alpha}$, which leads to
	\begin{equation}\label{BS_SVRG_n1}
		\begin{aligned}
			h^{\delta_t}(y_k) \leq{}& (1 - \tau_x) h^{\delta_t}(\tilde{x}_k) + \frac{\alpha^2 (1 - \tau_x)}{2L} \left(\norm{z_k - \xsd}^2 - \left(1 + \frac{\delta_t}{\alpha}\right)^2\Eik{\norm{z_{k+1} - \xsd}^2}\right) \\
			&+ \frac{\alpha + \delta_t - (\alpha + L+\delta_t)\tau_x}{L\delta_t} \innr{\phd{y_k}, \phd{\tilde{x}_k}}
			- \frac{\tau_x}{2L}\norm{\phd{y_k}}^2 - \frac{1 - \tau_x}{2L}\norm{\phd{\tilde{x}_k}}^2.
		\end{aligned}
	\end{equation}
	
	Substitute the choice $\tau_x = \frac{\alpha + \delta_t}{\alpha + L+\delta_t}$.
	\[
		h^{\delta_t}(y_k) \leq{} \frac{L}{\alpha + L + \delta_t} h^{\delta_t}(\tilde{x}_k) + \frac{\alpha^2}{2(\alpha +L + \delta_t)} \left(\norm{z_k - \xsd}^2 - \left(1 + \frac{\delta_t}{\alpha}\right)^2\Eik{\norm{z_{k+1} - \xsd}^2}\right).
	\]
	
	Note that by construction, $\E{h^{\delta_t}(\tilde{x}_{k+1})} = p\E{h^{\delta_t}(y_k)} + (1 - p)\E{h^{\delta_t}(\tilde{x}_k)}$, and thus
	\[
	\begin{aligned}
		\E{h^{\delta_t}(\tilde{x}_{k+1})} \leq{}& \left(1 - \frac{p(\alpha + \delta_t)}{\alpha + L + \delta_t}\right) \E{h^{\delta_t}(\tilde{x}_k)} \\&+ \frac{\alpha^2p}{2(\alpha +L + \delta_t)} \left(\E{\norm{z_k - \xsd}^2} - \left(1 + \frac{\delta_t}{\alpha}\right)^2\E{\norm{z_{k+1} - \xsd}^2}\right).
	\end{aligned}
	\]
	
	Since $\alpha$ is chosen as the positive root of $\left(1 - \frac{p(\alpha + \delta_t)}{\alpha + L + \delta_t}\right)\left(1 + \frac{\delta_t}{\alpha}\right)^2 = 1$, defining the potential function
	\begin{equation}\label{app:bs-svrg-potential}
	T_k \triangleq \E{h^{\delta_t}(\tilde{x}_k)} + \frac{\alpha^2p}{2\big(L + (1 -p)(\alpha + \delta_t)\big)}\E{\norm{z_k - \xsd}^2},
	\end{equation}
	we have $T_{k+1} \leq \left(1 + \frac{\delta_t}{\alpha}\right)^{-2} T_k$. 
	
	Thus, at iteration $k$, the following holds,
	\[
	\begin{aligned}
		\E{h^{\delta_t}(\tilde{x}_k)}\leq{}& \left(1 +\frac{\delta_t}{\alpha}\right)^{-2k} \left(h^{\delta_t}(x_0) + \frac{\alpha^2p}{2\big(L + (1 -p)(\alpha + \delta_t)\big)}\norm{x_0 - \xsd}^2\right)\\
		\leq{}& \left(1 +\frac{\delta_t}{\alpha}\right)^{-2k} \left(f^{\delta_t}(x_0) - f^{\delta_t}(\xsd) + \frac{\alpha^2p}{2\big(L + (1 -p)(\alpha + \delta_t)\big)}\norm{x_0 - \xsd}^2\right)\\
		\mleq{\star}{}& \left(1 +\frac{\delta_t}{\alpha}\right)^{-2k} \left(f(x_0) - f(\xs) + \frac{\alpha^2p}{2\big(L + (1 -p)(\alpha + \delta_t)\big)}\norm{x_0 - \xsd}^2\right),
	\end{aligned}
	\]
	where $\mar{\star}$ uses Lemma \ref{lem:regular} \textit{(ii)}.
	
	Note that using the interpolation condition \eqref{interpolation_sc}, we have
	\[
	\begin{aligned}
		\E{h^{\delta_t}(\tilde{x}_k)}\geq{}& \frac{1}{2L}\E{\norm{\phd{\tilde{x}_k}}^2}\\ ={}& \frac{1}{2L}\E{\norm{\pfd{\tilde{x}_k} - \delta_t(\tilde{x}_k - \xsd)}^2} \\ ={}& \frac{1}{2L}\E{\norm{\pf{\tilde{x}_k} + \delta_t(\tilde{x}_k - x_0) - \delta_t(\tilde{x}_k - \xsd)}^2} \\
		={}& \frac{1}{2L}\E{\norm{\pf{\tilde{x}_k} - \delta_t(x_0 - \xsd)}^2} \\
		\geq{}& \frac{1}{2L}\E{\norm{\pf{\tilde{x}_k} - \delta_t(x_0 - \xsd)}}^2.
	\end{aligned}
	\]
	
	Finally, we conclude that 
	\begin{equation}\label{P6}
		\begin{aligned}
			\E{\norm{\pf{\tilde{x}_k}}} &\leq{} \delta_t\norm{x_0 - \xsd} \\
			&+ \left(1 +\frac{\delta_t}{\alpha}\right)^{-k} \sqrt{2L\big(f(x_0) - f(\xs)\big) + \frac{L\alpha^2p}{L + (1 -p)(\alpha + \delta_t)}\norm{x_0 - \xsd}^2}.
		\end{aligned}
	\end{equation}
	
	\textbf{Under IDC:} Invoking Lemma \ref{lem:regular} \textit{(iii)} to upper bound \eqref{P6}, we obtain that for any $\xs \in \Xs$, 
	\[
	\E{\norm{\pf{\tilde{x}_k}}} \leq{} \left(\delta_t + \left(1 +\frac{\delta_t}{\alpha}\right)^{-k} \sqrt{L^2 + \frac{L\alpha^2p}{L + (1 -p)(\alpha + \delta_t)}}\right) \norm{x_0 - \xs}.
	\]
	
	\textbf{Under IFC:} Invoking Lemma \ref{lem:regular} \textit{(iv)} to upper bound \eqref{P6}, we can conclude that
	\[
	\E{\norm{\pf{\tilde{x}_k}}} \leq{} \left(\sqrt{2\delta_t} + \left(1 +\frac{\delta_t}{\alpha}\right)^{-k} \sqrt{2L + \frac{2L\alpha^2p}{\big(L + (1 -p)(\alpha + \delta_t)\big)\delta_t}}\right)\sqrt{f(x_0) - f(\xs)}.
	\]
	
\subsection{Proof to Theorem \ref{thm:R-Acc-SVRG-final-IDC}}
	\textit{(i)} At outer iteration $\ell$, if Algorithm \ref{alg:R-Acc-SVRG-G} breaks the inner loop (Step \ref{alg-step:IDC-break}) at iteration $k$, by construction, we have $ (1 + \frac{\delta_\ell}{\alpha})^{-k}\sqrt{C_\textup{IDC}} \leq \delta_\ell$
	. Then, from Proposition \ref{prop:R-Acc-SVRG-innr-loop} \textit{(i)},
	\[
	\E{\norm{\pf{\tilde{x}_k}}} \leq{} 2 \delta_\ell R_0 \mleq{\star} \epsilon q,
	\]
	where $\mar{\star}$ uses $\delta_\ell \leq \delta^\star_{\textup{IDC}}$. By Markov's inequality, it holds that
	\[
	\prob{\norm{\pf{\tilde{x}_k}} \geq \epsilon} \leq \frac{\E{\norm{\pf{\tilde{x}_k}}}}{\epsilon} \leq q,
	\]
	which means that with probability at least $1 -q$, Algorithm \ref{alg:R-Acc-SVRG-G} terminates at iteration $k$ (Step \ref{alg-step:terminate}) before reaching Step \ref{alg-step:IDC-break}.
	
	\textit{(ii)} Note that the expected gradient complexity of each inner iteration is $p (n + 2) + (1 - p)2 = np + 2$. Then, for an inner loop that breaks at Step \ref{alg-step:IDC-break}, its expected complexity is
	\[
	\E{\# \text{grad}_t} \leq (np+2) \left(\frac{\alpha}{\delta_t} + 1\right) \log{\frac{\sqrt{C_\textup{IDC}}}{\delta_t}}. 
	\]
	Substituting the choices in Proposition \ref{prop:R-params}, we obtain
	\[
	\E{\# \text{grad}_t} = O\left(\left(n + \sqrt{\frac{nL}{\delta_t}}\right)\log{\frac{L+\delta_t}{\delta_t}}\right).
	\]
	Thus, the total expected complexity before Algorithm \ref{alg:R-Acc-SVRG-G} terminates with high probability at outer iteration $\ell$ is at most (note that $\delta_t = \delta_0 / \beta^t$)
	\[
	\sum_{t = 0}^{\ell} {\E{\# \text{grad}_t}} = O\left(\left(\ell n + \frac{1}{\sqrt{\beta} - 1}\sqrt{\frac{nL\beta}{\delta_\ell}}\right)\log{\frac{L+\delta_\ell}{\delta_\ell}}\right).
	\]
	Since outer iteration $\ell>0$ is the first time $\delta_\ell \leq \delta^\star_\textup{IDC}$, we have $\delta_\ell \leq \delta^\star_\textup{IDC} \leq \delta_\ell \beta$. Moreover, noting that $\ell = O(\log{\frac{\delta_0}{\delta_\ell}})$ and $\delta_0 = L$, we can conclude that (omitting $\beta$)
	\[
	\begin{aligned}
		\sum_{t = 0}^{\ell} {\E{\# \text{grad}_t}} &= O\left(\left(n \log{\frac{\delta_0}{\delta_\ell}} + \sqrt{\frac{nL}{\delta_\ell}}\right)\log{\frac{L+\delta_\ell}{\delta_\ell}}\right)\\
		&= O\left(\left(n \log{\frac{L R_0}{\epsilon q}} + \sqrt{\frac{nLR_0}{\epsilon q}}\right)\log{\frac{LR_0}{\epsilon q} }\right).
	\end{aligned}
	\]

\subsection{Proof to Theorem \ref{thm:R-Acc-SVRG-final-IFC}}
	\textit{(i)} At outer iteration $\ell$, if Algorithm \ref{alg:R-Acc-SVRG-G} breaks the inner loop (Step \ref{alg-step:IFC-break}) at iteration $k$, by construction, we have $ (1 + \frac{\delta_\ell}{\alpha})^{-k}\sqrt{C_\textup{IFC}} \leq \sqrt{2\delta_\ell}$
	. Then, from Proposition \ref{prop:R-Acc-SVRG-innr-loop} \textit{(ii)},
	\[
	\E{\norm{\pf{\tilde{x}_k}}} \leq{} \sqrt{8 \delta_\ell \Delta_0} \mleq{\star} \epsilon q,
	\]
	where $\mar{\star}$ uses $\delta_\ell \leq \delta^\star_{\textup{IFC}}$. By Markov's inequality, it holds that
	\[
	\prob{\norm{\pf{\tilde{x}_k}} \geq \epsilon} \leq \frac{\E{\norm{\pf{\tilde{x}_k}}}}{\epsilon} \leq q,
	\]
	which means that with probability at least $1 -q$, Algorithm \ref{alg:R-Acc-SVRG-G} terminates at iteration $k$ (Step \ref{alg-step:terminate}) before reaching Step \ref{alg-step:IFC-break}.
	
	\textit{(ii)} Note that the expected gradient complexity of each inner iteration is $p (n + 2) + (1 - p)2 = np + 2$. Then, for an inner loop that breaks at Step \ref{alg-step:IFC-break}, its expected complexity is
	\[
	\E{\# \text{grad}_t} \leq (np+2) \left(\frac{\alpha}{\delta_t} + 1\right) \log{\sqrt{\frac{C_\textup{IFC}}{2\delta_t}}}. 
	\]
	Substituting the choices in Proposition \ref{prop:R-params}, we obtain
	\[
	\E{\# \text{grad}_t} = O\left(\left(n + \sqrt{\frac{nL}{\delta_t}}\right)\log{\frac{L}{\delta_t}}\right).
	\]
	Thus, the total expected complexity before Algorithm \ref{alg:R-Acc-SVRG-G} terminates with high probability at outer iteration $\ell$ is at most (note that $\delta_t = \delta_0 / \beta^t$)
	\[
	\sum_{t = 0}^{\ell} {\E{\# \text{grad}_t}} = O\left(\left(\ell n + \frac{1}{\sqrt{\beta} - 1}\sqrt{\frac{nL\beta}{\delta_\ell}}\right)\log{\frac{L}{\delta_\ell}}\right).
	\]
	Since outer iteration $\ell>0$ is the first time $\delta_\ell \leq \delta^\star_\textup{IFC}$, we have $\delta_\ell \leq \delta^\star_\textup{IFC} \leq \delta_\ell \beta$. Moreover, noting that $\ell = O(\log{\frac{\delta_0}{\delta_\ell}})$ and $\delta_0 = L$, we can conclude that (omitting $\beta$)
	\[
	\begin{aligned}
		\sum_{t = 0}^{\ell} {\E{\# \text{grad}_t}} &= O\left(\left(n \log{\frac{\delta_0}{\delta_\ell}} + \sqrt{\frac{nL}{\delta_\ell}}\right)\log{\frac{L}{\delta_\ell}}\right)\\
		&= O\left(\left(n \log{\frac{\sqrt{L\Delta_0}}{\epsilon q}} + \frac{\sqrt{nL\Delta_0}}{\epsilon q}\right)\log{\frac{\sqrt{L\Delta_0}}{\epsilon q}}\right).
	\end{aligned}
	\]

\section{Katyusha + L2S}
\label{app:l2s-katyusha}
By applying \textsf{AdaptReg} on Katyusha,  \citet{allen2017katyusha} showed that the scheme outputs a point $x_{s_1}$ satisfying $\E{f(x_{s_1})} - f(\xs)\leq \epsilon_1$ in \[O\left(n\log{\frac{LR_0^2}{\epsilon_1}} + \frac{\sqrt{nL}R_0}{\sqrt{\epsilon_1}}\right),\] oracle calls for any $\epsilon_1 > 0$ (cf. Corollary 3.5 in \citep{allen2017katyusha}). 

For L2S, \citet{li20a} proved that when using an $n$-dependent step size, its output $x_a$ satisfies (cf. Corollary 3 in \citep{li20a})
\[
	\E{\norm{\pf{x_a}}}^2\leq\E{\norm{\pf{x_a}}^2}  = O\left(\frac{\sqrt{n}L\big(f(x_0) - f(\xs)\big)}{T}\right),
\]     
after running $T$ iterations.

We can combine these two rates following the ideas in \citep{nesterov2012make}. Set $\epsilon_1 = O\big(\frac{T\epsilon^2}{\sqrt{n}L}\big)$ for some $\epsilon>0$ and let the input $x_0$ of L2S be the output $x_{s_1}$ of Katyusha. By chaining the above two results, we obtain the guarantee $\E{\norm{\pf{x_a}}} = O(\epsilon)$ in oracle complexity
\[
	O\left(n + T + n\log{\frac{n^{1/4}LR_0}{\sqrt{T}\epsilon}} + \frac{n^{3/4}LR_0}{\sqrt{T}\epsilon}\right).
\]

Minimizing the complexity by choosing $T = O\big(\frac{\sqrt{n}(LR_0)^{2/3}}{\epsilon^{2/3}}\big)$, we get the total oracle complexity
\[
O\left(n\log{\frac{LR_0}{\epsilon}} + \frac{\sqrt{n}(LR_0)^{2/3}}{\epsilon^{2/3}} \right).
\]
\end{document}